\documentclass{scrartcl}
\usepackage{amsmath}
\usepackage{amsfonts}
\usepackage{amssymb}
\usepackage{dsfont}
\usepackage[T1]{fontenc}
\usepackage[latin1]{inputenc}
\usepackage{epsfig}
\usepackage{graphicx}

\usepackage[boxed, lined]{algorithm2e}
\usepackage{float}
\usepackage{comment}
\usepackage{mathtools}

\usepackage{enumitem}

\usepackage{bm}

\newcommand{\subalign}[1]{%
   \vcenter{%
     \Let@ \restore@math@cr \default@tag
     \baselineskip\fontdimen10 \scriptfont\tw@
     \advance\baselineskip\fontdimen12 \scriptfont\tw@
     \lineskip\thr@@\fontdimen8 \scriptfont\thr@@
     \lineskiplimit\lineskip
     \ialign{\hfil$\m@th\scriptstyle##$&$\m@th\scriptstyle{}##$\hfil\crcr
       #1\crcr
     }%
   }%
}

\newcommand{\ba}{\mathbf{a}}
\newcommand{\bb}{\mathbf{b}}

\newcommand{\bx}{x}

\newcommand{\bz}{\mathbf{z}}
\newcommand{\bw}{\mathbf{w}}

\newcommand{\bH}{\mathbf{H}}

\newcommand{\D}{{\mathcal{D}}}
\newcommand{\A}{{\mathcal{A}}}
\newcommand{\B}{{\mathcal{B}}}
\newcommand{\W}{{\mathcal{W}}}
\newcommand{\bA}{{\mathbf{A}}}

\newcommand{\Nu}{{\mathcal{N}}}

\newcommand{\M}{{\mathcal{M}}}
\newcommand{\C}{{\mathcal{C}}}

\newcommand{\N}{\mathbb{N}}

\newcommand{\R}{\mathbb{R}}

\newcommand{\Rd}{\mathbb{R}^d}

\newcommand{\beq}{\begin{eqnarray*}}

\newcommand{\eeq}{\end{eqnarray*}}

\newcommand{\beqm}{\begin{eqnarray}}

\newcommand{\eeqm}{\end{eqnarray}}

\newtheorem{theorem}{Theorem}

\newtheorem{lemma}{Lemma}
\newtheorem{definition}{Definition}

\DeclareOldFontCommand{\bf}{\normalfont\bfseries}{\mathbf}
\DeclareOldFontCommand{\it}{\normalfont\itshape}{\mathit}

\newcommand{\EXP}{{\mathbf E}}
\newcommand{\PROB}{{\mathbf P}}

\renewcommand{\P}{{\cal P}}
\renewcommand{\bf}{\normalfont \bfseries}
\renewcommand{\it}{\normalfont \itshape}

\allowdisplaybreaks

\DeclareMathOperator*{\esssup}{ess\,sup}
\DeclareMathOperator*{\supp}{supp}

\begin{document}
\renewcommand{\thefootnote}{\fnsymbol{footnote}}
\newcommand{\F}{{\cal F}}
\newcommand{\Sp}{{\cal S}}
\newcommand{\G}{{\cal G}}
\newcommand{\HH}{{\cal H}}

\newcounter{constcounter}

\makeatletter
\newcommand{\const}[1][\relax]{%
  \ifmeasuring@
    c_{?}%
  \else
    \ifx#1\relax
      \stepcounter{constcounter}%
      c_{\theconstcounter}%
    \else
      \@ifundefined{const@#1}{%
        \stepcounter{constcounter}%
        \expandafter\xdef\csname const@#1\endcsname{\theconstcounter}%
        c_{\csname const@#1\endcsname}%
      }{%
        c_{\csname const@#1\endcsname}%
      }%
    \fi
  \fi
}
\makeatother


\begin{center}

  {\LARGE \bf
    Estimation of a regression function from dependent data
    by over-parametrized deep neural networks learned by gradient descent
  }
\footnote{
  Running title: {\it Estimation of a regression function from
    dependent data}}
\vspace{0.5cm}

Michael Kohler$^{1}$, 
Adam Krzy\.zak$^{2,}$\footnote{Corresponding author. Tel:
  +1-514-848-2424 ext. 3007, Fax:+1-514-848-2830}
and Vincent Molinero R\"omer$^{1}$
\\

{\it $^1$
Fachbereich Mathematik, Technische Universit\"at Darmstadt,
Schlossgartenstr. 7, 64289 Darmstadt, Germany,
email: kohler@mathematik.tu-darmstadt.de,
roemer@mathematik.tu-darmstadt.de}

{\it $^2$ Department of Computer Science and Software Engineering, Concordia University, 1455 De Maisonneuve Blvd. West, Montreal, Quebec, Canada H3G 1M8, email: krzyzak@cs.concordia.ca}

\end{center}
\vspace{0.5cm}

\begin{center}
April 1, 2026
\end{center}
\vspace{0.5cm}

\noindent
    {\bf Abstract}\\
    Estimation of a regression function from exponentially $\beta$-mixing data
is considered. The $L_2$ error with integration with
    respect to the design is used as the error criterion.
    Deep neural network estimates with logistic activation function
 are defined, where all parameters are learned by gradient descent. 
 The rate of convergence of the expected $L_2$ error is analyzed for $(p,C)$-smooth  regression functions.
In the special case that the design is concentrated
    on a $d^*$-dimensional manifold, it is shown that the expected $L_2$ error
    of the estimate achieves a rate of convergence which depends on $d^*$
    and not on the dimension $d$ of the design.
    
    \vspace*{0.2cm}

\noindent{\it AMS classification:} Primary 62G08; secondary 62G20.

\vspace*{0.2cm}

\noindent{\it Key words and phrases:}
Deep neural networks,
dependent data,
dimension reduction,
gradient descent,
over-parametrization,
rate of convergence,
regression estimation.

\section{Introduction}
\label{se1}
Deep neural networks are nowadays among the most successful statistical methods 
applied in practice, e.g. in image classification 
(cf., e.g., Krizhevsky, Sutskever and Hinton  (2012)),
in language recognition (cf., e.g.,  Kim (2014))
in machine translation (cf., e.g., Wu et al. (2016))
in game playing (cf., e.g., Silver et al. (2017))
or in simulation of human conversation (cf., e.g., Minaee et al. (2025)).
Motivated by these successes in applications, deep learning was
also intensively studied theoretically in the past ten years.
Often this is done in the context of nonparametric regression, where
a sample
\begin{equation}
\label{se1eq1}
\D_n = \{ (X_1,Y_1), \dots, (X_n,Y_n) \}
\end{equation}
of a $\R^d \times \R$--valued random vector $(X,Y)$ satisfying
$\EXP\{Y^2\} < \infty$ is given, and the task is to construct an
estimate
\[
m_n(\cdot)=m_n(\cdot, \D_n): \R^d \rightarrow \R
\]
of the corresponding regression function $m: \R^d \rightarrow \R$,
$m(x)=\EXP\{Y|X=x\}$
such that the so--called $L_2$ error 
\[
\int |m_n(x)-m(x)|^2 \PROB_X (dx)
\]
is small (see Chapter 1 in Gy\"orfi et al. (2002) for a motivation
for using the $L_2$ error as the error criterion in nonparametric
regression).

It is well-known that the rate of convergence of the $L_2$ error might
be arbitrarily slow in case that one does not impose smoothness
assumptions on the regression function (cf., e.g., Chapter 3 in
Gy\"orfi et al. (2002)). In the sequel we will assume that the
regression function is $(p,C)$--smooth according to the following definition.
\begin{definition}
\label{intde1} 
  Let $p=q+s$ for some $q \in \N_0$ and $0< s \leq 1$.
A function $m:\R^d \rightarrow \R$ is called
{\em $(p,C)$-smooth}, if for every $\bm{\alpha}=(\alpha_1, \dots, 
\alpha_d) \in \N_0^d$
with $\sum_{j=1}^d \alpha_j = q$ the partial derivative
$\partial^q m/(\partial x_1^{\alpha_1}
\dots
\partial x_d^{\alpha_d}
)$
exists and satisfies
\[
\left|
\frac{
\partial^q m
}{
\partial x_1^{\alpha_1}
\cdots
\partial x_d^{\alpha_d}
}
(\bx)
-
\frac{
\partial^q m
}{
\partial x_1^{\alpha_1}
\cdots
\partial x_d^{\alpha_d}
}
(\bz)
\right|
\leq
C \cdot
\| \bx-\bz \|^s
\]
for all $\bx,\bz \in \R^d$, where $\Vert\cdot\Vert$ denotes the Euclidean norm.
\end{definition}
Stone (1982) showed that the optimal minimax rate of convergence
for estimation of $(p,C)$--smooth regression functions in the case of
i.i.d. data is
\[
n^{-\frac{2p}{2p+d}}.
\]
This rate suffers from the so--called ''curse of dimensionality'': it becomes slow when 
$d$ is large relative to $p$.

If the data $\D_n$ is an i.i.d. sample
independent of $(X,Y)$, then least squares estimates based on deep
neural networks can achieve very fast rate of convergence results even for
large $d$. The reason behind this is the structure of the deep neural networks which
enables us to use approximation results for smooth functions by neural networks
(cf., e.g., Yarotsky (2018)) to approximate the composition of smooth functions 
by deep neural networks. Here dimension reduction occurs in the
sense that the rate of convergence 
of the corresponding least squares regression estimates
does not depend on the dimension
$d$ of $X$, but on the maximum number of variables in the individual
functions to be composed, and therefore the estimates achieve good 
rate of convergence results under such a hierarchical composition
model for the regression function even for very large $d$. The first paper
showing that deep neural networks can circumvent the
'curse of dimensionality' in the above setting was Kohler and
Krzy\.zak (2017), and later on this was extended by many additional
results (see, e.g., Bauer and Kohler (2019), Schmidt-Hieber (2020) and Kohler
and Langer (2021)). Another setting where it was shown that deep
neural networks are able to circumvent the curse of dimensionality
is manifold case, where it is assumed that $X$ takes on values
on some $d^*$ dimensional manifold defined as follows:
\begin{definition}
  \label{intde2} 
  Let $\M \subseteq \Rd$ be compact and let $d^* \in \{1, \dots, d\}$.

  \noindent
      {\bf a)} We say that $U_1,\dots,U_r$ is
      an {\em open covering  of $\M$}, if $U_1,\dots,U_r \subset \Rd$
      are open (with respect to the Euclidean topology on $\Rd$)
      and satisfy
      \[
\M \subseteq \bigcup_{l=1}^r U_l.
\]

  \noindent
      {\bf b)} We say that
\[
\psi_1, \dots, \psi_r: [0,1]^{d^*}\rightarrow \Rd
\]
are {\em bi-Lipschitz functions}, if there exists $0 < C_{\psi,1} 
\leq C_{\psi,2}
< \infty$ such that
\begin{equation}
\label{de2eq1}
C_{\psi,1}  \cdot \|\bx_1-\bx_2\|
\leq
\| \psi_l(\bx_1)-\psi_l(\bx_2) \|
\leq
C_{\psi,2}  \cdot \|\bx_1-\bx_2\|
\end{equation}
holds for any $\bx_1,\bx_2 \in  [0,1]^{d^*}$ and any
$l \in \{1, \dots, r\}$.

        \noindent
            {\bf c)}
            We 
say that $\M$ is a {\em $d^*$-dimensional Lipschitz-manifold} if 
there exist {\em bi-Lipschitz functions} $\psi_i : [0,1]^{d^*} \to \R^d$ $(i \in \{
  1,\dots,r \})$, 
and an open covering $U_1,\dots,U_r$ of $\M$ such that 
\[
\psi_l\left( (0,1)^{d^*} \right) = \M \cap U_l
\]
holds for all $l \in \{1, \dots, r\}$.
Here we call $\psi_1, \dots, \psi_r$ the {\em parametrizations}
of the manifold.
\end{definition}
In this case Kohler, Langer and Reif (2023)
and Jiao et al. (2023) showed that deep neural networks are able to achieve a rate
of convergence which depends on $d^*$ rather than on $d$.

In many applications, it is not possible to observe an i.i.d. sample
of $(X,Y)$. In the past few years several results were shown which
demonstrate that the above results also holds for dependent data
sets $\D_n$ which are exponentially $\beta$-mixing according to the following
definition.

\begin{definition}
	\label{intde3}\phantom{P}
	{\bf a)}
	Let $N,K\in\N$, let
	$X:\Omega\to\R^N$ and
	$Y:\Omega\to\R^K$
	be Borel measurable random variables defined on the same
        probability space.
	The $\beta$-mixing coefficient between $X$ and $Y$ is defined by
	\[
	\beta(X,Y)
	=\EXP\{\esssup_{C\in \sigma(Y)} |\PROB(C|X)-\PROB(C)|\}
	=\EXP\{\esssup_{A\in \B_K} |\PROB(Y\in A|X)-\PROB(Y\in A)|\},
	\]
	where $\sigma(Y)$ is the $\sigma$-field generated by $Y$.

	\noindent
	{\bf b)}
	A sequence of random variables $X_1, X_2, \dots$ is called $\beta$-mixing if
	$$
	{\beta}_s(X_1, X_2, \dots)=\sup_{k\in\N} {\beta}((X_1,\dots, X_k), X_{k+s})
	\to 0 \hspace{0.5em}(s\to\infty).
	$$

\noindent
{\bf c)}
We say that $(X_1,Y_1)$, $(X_2,Y_2)$, \dots is exponentially
$\beta$-mixing
if there exists constants $\const[new1], \const[new2]>0$ such that
\[
{\beta}_s((X_1,Y_1), (X_2,Y_2), \dots)
\leq \const[new1] \cdot e^{- \const[new2] \cdot s} \quad (s \in \N),
\]
i.e., if
\begin{eqnarray*}
&&
\sup_{k\in\N} 
\EXP\{\esssup_{A\in \B_{d+1}
} 
| \PROB\{   (X_{k+s},Y_{k+s}) \in A| (X_1,Y_1), \dots, (X_k,Y_k)\}
-
\PROB\{ (X_{k+s},Y_{k+s}) \in A\}|
\}
\\
&&
\leq \const[new1] \cdot e^{- \const[new2] \cdot s} \quad (s \in \N).
\end{eqnarray*}
\end{definition}
It was shown in Ma and Safikhani (2022) that least squares estimates
based on deep neural networks achieve dimension reduction 
also in the case of exponentially $\beta$--mixing data provided the
regression function satisfies a hierarchical composition model.
In a time series setting a related result, which covers in addition
the case of manifold data, has been obtained by
Padilla et al. (2024).

The above least squares estimates based on deep neural networks cannot
be computed in applications since the computation of the least squares
estimates requires minimization of the empirical risk of the
networks, which is a nonlinear and non-convex function of the weights. Instead
one computes these estimates approximately by using gradient descent.
Here, even in i.i.d case, it is not known whether the corresponding 
estimates achieve dimension reduction in case that the regression 
function satisfies a hierarchical composition model or in case that the predictor
takes on values on some submanifold of $\R^d$ as long as the
size of the network or the number of gradient descent steps are not
exponentially growing in the sample size. 

In this article we focus on the case where the regression function is
$(p,C)$--smooth.
In this situation it was shown in Kohler (2026) that there exists a deep
neural network estimate
(with logistic activation function)
 learned by gradient descent, where both the
size of the network and the number of gradient descent steps are
bounded by a polynomial in the data size $n$, which satisfies for any
$\epsilon>0$
\begin{equation}
\label{se1eq2}
\EXP \int |m_n(x)-m(x)|^2 \PROB_X(dx)
\leq \const[new3] \cdot n^{-\frac{2p}{2p+d}+\epsilon}
\end{equation}
for some $\const[new3]=\const[new3](\epsilon) >0$
in case where the data are i.i.d, that $X$ takes on with probability one 
values in some bounded set, that the regression function is $(p,C)$--smooth 
and that the distribution of $Y$ is subgaussian (cf., (A1) below).

In this article we use the following model for our data:
We assume that
 \[
(X,Y), \,(X_1,Y_1),\,(X_2,Y_2),\, \dots
\]
are $\R^d\times \R$-valued random variables defined on the same probability space satisfying
\begin{itemize}
\item[(A1)] $Y$ is subgaussian, i.e., $\EXP\left\{e^{\const[c1th1]\cdot Y^2}\right\}<\infty$ for some $\const[c1th1]>0$,
\item[(A2)] $\supp(X)$ is bounded,
\item[(A3)] $m(x)=\EXP\{Y|X=x\}$ is $(p,C)$-smooth for some $p,C>0$,
\item[(A4)] $(X,Y), \,(X_1,Y_1),\, \dots$ are identically distributed,
\item[(A5)] $(X,Y)$ is independent of $(X_1,Y_1),\,(X_2,Y_2),\, \dots$,
\item[(A6)] $(X_1,Y_1),\,(X_2,Y_2),\, \dots$ is
exponentially $\beta$-mixing.
\end{itemize}
Our main results in this article extend the above-described result
from Kohler (2026) in two ways:
First, we show that (\ref{se1eq2}) also holds in case that the data
is exponentially $\beta$-mixing. Second, we consider the case that
$X$ takes on values on some $d^*$-dimensional manifold.
In this case we show that our deep neural network regression estimate
learned by gradient descent satisfies for any $\epsilon>0$
\[
\EXP \int |m_n(x)-m(x)|^2 \PROB_X(dx)
\leq \const[new4] \cdot n^{-\frac{2p}{2p+d^*}+\epsilon}
\]
for some $\const[new4]=\const[new4](\epsilon)>0$
in case that the data is exponentially $\beta$-mixing, that $X$ takes on
with probability one values in some $d^*$--dimensional manifold,
that the regression function is $(p,C)$--smooth and that the
distribution of $Y$ is subexponential. 

\subsection{Discussion of related results}

In recent years, a large number of research studies were devoted to analysis of regression estimates based on deep neural networks and trained 
by independent, identically distributed (i.i.d.) data. 
In order to establish theoretical guarantees for these estimates one needs to 
simultaneously study approximation, estimation and optimization issues.
 
Approximation abilities of neural network regression estimates have been 
investigated by Yarotsky (2017, 2018),  Yarotsky and Zhevnerchuk (2019), 
Lu et al. (2020), Langer (2020) and in the literature cited therein. 
The authors studied expressive powers of shallow and deep neural networks 
regression estimates with specific activation functions such 
as ReLU.

Generalization aspects of deep neural network regression estimates have been 
established by using the classical VC theory 
to bound the VC dimension of classes of neural networks, see Bartlett et al. (2017), 
Bartlett et al. (2019) and by means of bounding the Rademacher complexity, 
see,  e.g., Bartlett and Mendelson (2002), Liang, Rakhlin and Sridharan (2015), 
Golowich, Rakhlin and Shamir (2019), Lin and Zhang (2019) and Wang and Ma (2022).

The asymptotic properties of the neural network regression estimates minimizing 
least squares have been studied by Kohler and Krzy\.zak (2017).
They obtained the rate of convergence independent of the input dimension 
within the class of regression functions following the hierarchical interaction model 
with additional smoothness constraints. These results were later extended to arbitrary 
smooth functions by Bauer and Kohler (2019). Furthermore, Schmidt-Hieber (2020) 
proved that the least squares neural network regression estimates with ReLU activations that satisfy certain 
sparsity constraints achieve the minimax rates of convergence up to logarithmic factors 
for a more general class of functions called hierarchical composition model. 
Kohler and Langer (2021) obtained the same rate for a linear combination of
fully connected networks without sparsity constraints. 
Further extensions of these results were obtained by Imaizumi and Fukamizu (2018), 
Suzuki (2018) and Suzuki and Nitanda (2019).  

Solving least squares problem for deep neural network regression estimation in practice 
is prohibitively expensive. Alternatively, gradient descent has been established 
as a method of choice for optimizing the weights in deep neural networks. 
Zou  et al. (2018), Du et al. (2019), Allen-Zhu, Li and Song (2019) 
and Kawaguchi and Huang (2019) proved that gradient descent applied to over-parameterized deep
neural networks yields neural networks globally minimizing the empirical risk, 
however Kohler and Krzy\.zak (2021) showed the resulting estimates do  not generalize well. 
Cao and Gu (2019) and Chen et al. (2021) 
proved generalization ability of the neural network estimates 
trained by gradient descent, but they did not analyze their approximation abilities.

In order to establish theoretical guarantees for the estimates trained by 
gradient descent the approximation, generalization and optimization errors 
have to be studied simultaneously. Such approach has been used by Braun et al. (2024) 
for shallow (or one hidden layer) neural network and it led to dimension-free rate of 
convergence for regression functions satisfying Barron (1993, 1994) condition, i.e., 
for functions whose Fourier transform has a finite first moment. 
In the case of deep neural networks with multiple hidden layers Kohler and Krzy\.zak (2025a) 
used over-parametrized deep neural networks defined as
a linear combination of a 
huge number of deep neural networks computed in parallel with randomly 
initialized weights and they applied gradient descent to effectively select 
a subset of the neural networks within the linear combination. 
They apply metric entropy bounds, see Birman and Solomjak (1967) 
and Li, Gu and Ding (2021), to 
establish generalization abilities of over-parametrized neural
networks and show the rate of convergence of order close to
$n^{-1/(1+d)}$, whereas for interaction models the rate becomes
$n^{1/(1+d^*)}$, i.e., it is independent of dimension, where $d^*$ is
effective dimension of the interaction model. The above results are 
valid for $(p,C)$--smooth
 regression functions with $p=1/2$.
In Kohler (2026) over-parametrized deep neural neural network
estimates learned by gradient descent have been analyzed in the case of a
$(p,C)$--smooth regression function with general $p$.
An interesting by-product of these studies is that learning of inner weights 
of the deep network regression estimate is not important. E.g., Gonon (2021) 
does not train inner weights at all, whereas Braun et al. (2024), 
Kohler and Krzy\.zak (2025a) and Kohler (2026) use the fact that the relevant inner 
weights remain close to their starting values during gradient descent and they use 
gradient descent to only  train output weights of $L_2$ regularized network.  
Similar approach has also been taken by Andoni et al. (2014), Daniely (2017),  
Huang, Chen and Siew (2006), Rahimi and Recht (2008a), Rahimi and Recht (2008b) 
and Rahimi and Recht (2009). 
Universal consistency of over-parametrized deep learning regression estimates 
trained by gradient descent has been demonstrated by Drews and Kohler (2024). 
Statistically guided deep learning has been investigated by Kohler and Krzy\.zak (2025b), 
who have established a bound on the expected $L_2$ error of a deep neural network 
regression estimate. 

There exist alternative approaches for analyzing gradient descent training. 
One is the neural tangent kernel proposed by 
Jacot, Gabriel and Hongler (2020) and studied by Hanin and Nica (2019), 
Wilson et al. (2025) and by Mahowald et al (2026). Another approach 
is the  mean-field approach, cf., Mei, Montanari and Nguyen (2018), 
Chizat and Bach (2018) (further extended by Wojtowytsch (2020)) 
and Nguyen and Pham (2020). 
A comprehensive survey of over-parametrized deep neural network
estimates trained by gradient descent is presented in 
Bartlett, Montanari and Rakhlin (2021).

Recently introduced deep transformer networks have been very successful 
in large language models such as Chat GPT and in image and video processing 
and they became very popular in research and applications. 
They were introduced by Vaswani et al. (2017) and their approximation 
and generalization abilities were established by Gurevych, Kohler and Sahin (2022). 
The rates of convergence of over-parametrized transformer classifiers learned by 
gradient descent have been obtained by Kohler and Krzy\.zak (2023).    

All results presented above concern the case of independent, 
identically distributed (i.i.d.) data. In recent years a number 
of results for deep learning with dependent data have become available. 
The most common notions of dependence are so-called mixing coefficients. 
The $\alpha$-mixing has been introduced by Rosenblatt (1956).  
The $\beta$-mixing has been defined by Kolmogorov, but first appeared in 
the paper by Wolkonski and Rozanov (1959). Ibragimov (1962) introduced the 
$\phi$-mixing coefficient. Blum, Hanson and Koopmans (1963) introduced 
the $*$-mixing coefficient. Properties of various dependencies were proven 
in Doob (1963). A general survey of mixing coefficients and examples is presented 
by Doukhan (1994). Barrera and Gobet (2021) generalize uniform deviation inequalities 
for the empirical process from the independent data to dependent $\beta$-mixing case. 
Their results make it possible to analyze errors of the least-squares regression 
schemes for dependent data. 
Kengne and Wade (2025a) studied deep learning with weakly dependent data and 
Kengne and Wade (2025b) considered strongly mixing observations. A general framework 
for deep learning with dependent data was proposed by Kengne and Wade (2025c). 
They obtained rates of convergence for dependencies described by $C$-mixing processes, 
strong mixing processes and $\phi$-mixing processes. Alquier and Kengne (2025) 
proved minimax optimality of deep neural network estimates with ReLU activation 
for the least squares regression problem and Markov chain dependence. 
Ma and Safikhani (2022) established the non-asymptotic bounds for prediction 
errors of sparse neural networks with ReLU activation functions under various 
mixing conditions including $\alpha$-mixing process and autoregressive time-series 
process. None of the above papers considers deep neural network estimates learned 
by gradient descent. In the present paper we analyze deep neural network estimates 
learned by gradient descent in the case of exponential $\beta$-mixing data.

\subsection{Notation}
\label{se1sub5}
  The sets of natural numbers, real numbers and nonnegative real numbers
  are denoted by $\N$, $\R$ and $\R_+$, respectively.
  For $z \in \R$, we denote
the smallest integer greater than or equal to $z$ by
$\lceil z \rceil$. And the largest integer less than or equal to $z$
is denoted by $\lfloor z \rfloor$.
The Euclidean norm of $x \in \Rd$
is denoted by $\|x\|$, and the scalar product of $x,y \in \R^d$ is denoted
by $<x,y>$. 
For $f:\R^d \rightarrow \R$
\[
\|f\|_\infty = \sup_{x \in \R^d} |f(x)|
\]
is its supremum norm.
If $(X_i)_{i \in I}$ is a nonempty family of real valued random
variables,
we define $\esssup_{i \in I} X_i$ as the (almost surely) unique
variable $Y$ which satisfies
\begin{itemize}
\item[(i)] $ Y \geq X_i$ $a.s.$ for every $i \in I$,
\item[(ii)] for any variable $\tilde{Y}$ which satisfies (i) we have
  $Y \leq \tilde{Y}$ $a.s.$
\end{itemize}

A finite collection $f_1, \dots, f_N:\Rd \rightarrow \R$
  is called an $L_p$ $\varepsilon$--covering of $\F$ on $x_1^n$
  if for all $f \in \F$
  \[
  \min_{1 \leq j\leq N}
  \left(
  \frac{1}{n} \sum_{k=1}^n |f(x_k)-f_j(x_k)|^p
  \right)^{1/p} \leq \varepsilon
  \]
  hold.
  The $L_p$ $\varepsilon$--covering number of $\F$ on $x_1^n$
  is the  size $N$ of the smallest $L_p$ $\varepsilon$--covering
  of $\F$ on $x_1^n$ and is denoted by $\Nu_p(\varepsilon,\F,x_1^n)$.

For $z \in \R$ and $\kappa \geq 0$ we define
$T_\kappa z = \max\{-\kappa, \min\{\kappa,z\}\}$. If $f:\R^d \rightarrow
\R$
is a function  then we set
$
(T_{\kappa} f)(x)=
T_{\kappa} \left( f(x) \right)$, and if $\F$ is a class of functions
$f:\R^d \rightarrow \R$ we set
$T_\kappa \F = \{ T_\kappa f \, : \, f \in \F \}$.

\subsection{Outline}
\label{se1sub6}
Section \ref{se2} contains the definition of the estimate.
The main results are presented in Section \ref{se3}  and proven
in Section \ref{se4}. Auxiliary results concerning mixing are
presented in the appendix.

\section{Definition of the estimate}
\label{se2}
We use the so--called logistic squasher
$\sigma(x)=1/(1+e^{-x})$
as our activation function
throughout the paper.
The topology we use is the same as in Kohler (2026) and is defined as follows:
For parameters $K_n, L, r \in \N$ we set
\begin{equation}\label{se2eq1}
f_\bw(x) = \sum_{j=1}^{K_n} w_{1,1,j}^{(L)} \cdot f_{j,1}^{(L)}(x)
\end{equation}
for some $w_{1,1,1}^{(L)}, \dots, w_{1,1,K_n}^{(L)} \in \mathbb{R}$, where
$f_{j,1}^{(L)}=f_{\bw,j,1}^{(L)}$ are recursively defined by
\begin{equation}
  \label{se2eq2}
f_{k,i}^{(l)}(x) = 
f_{\bw,k,i}^{(l)}(x) = 
\sigma\left(\sum_{j=1}^{r} w_{k,i,j}^{(l-1)}\cdot f_{k,j}^{(l-1)}(x) + w_{k,i,0}^{(l-1)} \right)
\end{equation}
for some $w_{k,i,0}^{(l-1)}, \dots, w_{k,i, r}^{(l-1)} \in \mathbb{R}$
$(l=2, \dots, L)$
and
\begin{equation}
  \label{se2eq3}
f_{k,i}^{(1)}(x) = 
f_{\bw,k,i}^{(1)}(x) = 
\sigma \left(\sum_{j=1}^d w_{k,i,j}^{(0)}\cdot x^{(j)} + w_{k,i,0}^{(0)} \right)
\end{equation}
for some $w_{k,i,0}^{(0)}, \dots, w_{k,i,d}^{(0)} \in \mathbb{R}$.

\noindent
That is, we consider neural networks consisting of $K_n$ parallel fully connected 
subnetworks of depth $L$ and width $r$, with the final output given by a linear combination of their outputs.
We write
$(w_{k,i,j}^{(l)})_{i,j,l}$ for the weights in the $k$-th subnetwork, where
$w_{k,i,j}^{(l)}$ denotes the weight between neuron $j$ in layer
$l$ and neuron $i$ in layer $l+1$.

The weights $\bw^{(0)}=((\bw^{(0)})_{k,i,j}^{(l)})_{k,i,j,l}$ are first initialized as follows: 
Set
\begin{equation}
\label{se2eq4}
(\bw^{(0)})_{1,1,k}^{(L)}=0
\quad (k=1, \dots, K_n),
\end{equation}
 choose the weights of layer $l \in \{1, \dots, L-1\}$ $(\bw^{(0)})_{k,i,j}^{(l)}$ uniformly distributed on
$[-\const[c1], \const[c1]]$, and
choose the weights of layer $0$
$(\bw^{(0)})_{k,i,j}^{(0)}$ uniformly distributed on
\[
  [-\const[c2] \cdot (\log n) \cdot n^\tau, \const[c2] \cdot (\log n) \cdot n^\tau],
  \]
  for parameters $\const[c1], \const[c2], \tau>0$.
We assume that all components of $\bw^{(0)}$ chosen in this way are independent.

Subsequently, we perform $t_n\in\N$ gradient descent steps of step size $\lambda_n>0$ in an attempt to minimize the empirical $L_2$ risk
\begin{equation}
\label{se2eq5}
F_n(\bw)
= \frac{1}{n} \sum_{i=1}^n | Y_i - f_\bw (X_i)|^2.
\end{equation}
Therefore, we set 
\begin{equation}
\label{se2eq6}
\bw^{(t+1)}
=
\bw^{(t)}
-
\lambda_n \cdot \nabla_{\bw} F_n(\bw^{(t)})
\quad
(t=0, \dots, t_n-1).
\end{equation}
Note that the choice of the weights implies
\begin{align*}
    F_n(\bw^{(0)})=\frac{1}{n}\sum_{i=1}^n|Y_i|^2.
\end{align*}
The final estimate $m_n$ is a truncated version of the neural network with weights according to the weight vector $\bw^{(t_n)}$, i.e.
\begin{equation}
\label{se2eq7}
m_n(x)= T_{\kappa_n} (f_{\bw^{(t_n)}}(x)),
\end{equation}
 where $\kappa_n = \const[c2th1] \cdot \log n\hspace{0.5em}(\const[c2th1]>0)$ and $T_{\kappa} z
= \max\{ \min\{z, \kappa\}, - \kappa\}$ for $z \in \R$
and $\kappa>0$.

\section{Main results}
\label{se3}
Our main results are presented in the following theorem.

\begin{theorem}
  \label{th1}
  Assume (A1) - (A6) holds, set $p=q+s$ where $q\in \N_0$ and $s\in (0,1]$, 
  set $\kappa_n = \const[c2th1]\cdot \log n$ for some $\const[c2th1]>0$ 
  and assume  $\const[c2th1] \cdot \const[c1th1] \geq 3$. 
  Define the estimate $m_n$ as in Section \ref{se2} 
  and assume $\const[c1], \const[c2] >0$ are sufficiently large.
  
  \noindent
  {\bf a)} Set
  \[
  L=\lceil \log_2(q+d) \rceil+1,
  \quad r=4 \cdot \lceil (p+d)^2 \rceil
  \quad \text{and}\quad
  \tau=\frac{1}{2p+d}
  \]
  and choose
  $K_n \in \N$ such that for some $\const[th1const1] >0$
  \[
  \frac{K_n}{n^{\const[th1const1]}} \rightarrow 0 \quad (n \rightarrow \infty)
  \]
  and
  \begin{equation}\label{th1eq1}
  \frac{K_n}{n^{
((2p+2d)\tau + 1.5)\, r((r+1)(L-1) + (d+1))
+
\tau (r(d+1) + 4(p+d))
+
2
  }}
  \rightarrow \infty \quad (n \rightarrow \infty).
  \end{equation}
  Set
\[
\lambda_n=\frac{\const[c3th1]}{K_n^{3/2}\cdot \kappa_n}
\quad \text{and}\quad
t_n=\left\lceil
\const[c4th1] \cdot \frac{K_n^{3/2}}{\kappa_n}
\right\rceil
\]
for some $\const[c3th1], \const[c4th1] >0$.
Then we have for any $\epsilon>0$
\[
\EXP \int | m_n(x)-m(x)|^2 \PROB_X (dx)
\leq \const[th1const2] \cdot n^{- \frac{2p}{2p+d} + \epsilon}
\]
for some $\const[th1const2]=\const[th1const2](\epsilon)>0$.

\noindent
{\bf b)}
  Let $\M\subseteq \R^d$ be a $d^*$-dimensional Lipschitz-manifold and assume $\supp(X)\subseteq \M$.
  Set $L,r$ as in part a) and set
  \[
  \tau=\frac{1}{2p+d^*}.
  \]
 
 Choose $K_n$, $\lambda_n$ and $t_n$ as in part a) with this new value for $\tau$.
Then we have for any $\epsilon>0$
\[
\EXP \int | m_n(x)-m(x)|^2 \PROB_X (dx)
\leq \const[th1const3] \cdot n^{- \frac{2p}{2p+d^*} + \epsilon}
\] 
for some $\const[th1const3]=\const[th1const3](\epsilon)>0$.
  \end{theorem}

\noindent
{\bf Remark 1.} We say that $(X_1,Y_1)$, $(X_2,Y_2)$, \dots is subexponentially
$\beta$-mixing
if there exists constants $\const[new5], \const[new6]>0$ and $\delta \in (0,1)$ such that
\[
{\beta}_s((X_1,Y_1), (X_2,Y_2), \dots)
\leq \const[new5] \cdot e^{- \const[new6] \cdot s^\delta} \quad (s \in \N),
\]
i.e., if
\begin{eqnarray*}
&&
\sup_{k\in\N} 
\EXP\{\esssup_{A\in \B_{d+1}
} 
| \PROB\{   (X_{k+s},Y_{k+s}) \in A| (X_1,Y_1), \dots, (X_k,Y_k)\}
-
\PROB\{ (X_{k+s},Y_{k+s}) \in A\}|
\}
\\
&&
\leq \const[new5] \cdot e^{- \const[new6] \cdot s^\delta} \quad (s \in \N).
\end{eqnarray*}
It follows from the proof of Theorem \ref{th1} that the assertion
also holds if the sample is subexponential $\beta$-mixing. To prove
this it suffices to choose
\[
N_n = \lceil (\log n)^{2/\delta} \rceil
\]
in the proof of Lemma \ref{lea3}.

\noindent
{\bf Remark 2.}  Let $(X_t)_{t \in \N}$ be a sequence of
exponentially $\beta$-mixing $\R^d$--valued
random variables, let $(\epsilon_t)_{t \in \N}$ be a sequence
of independent identically distributed square integrable
real-valued random variables
with expectation  zero which are independent of $(X_t)_{t \in \N}$,
let $m: \R^d \rightarrow \R$ be a Borel-measurable function
such that $\EXP\{|m(X_t)|^2\}<\infty$, and set
\[
Y_t = m(X_t) + \epsilon_t \quad (t \in \N).
\]
Then $((X_t,Y_t))_{t \in \N}$ is exponentially $\beta$-mixing since
\begin{eqnarray*}
&&
\sup_{k\in\N} 
\EXP\{\esssup_{A\in \B_{d+1}
} 
| \PROB\{   (X_{k+s},Y_{k+s}) \in A| (X_1,Y_1), \dots, (X_k,Y_k)\}
-
\PROB\{ (X_{k+s},Y_{k+s}) \in A\}|
\}
\\
&&
\leq 
\sup_{k\in\N} 
\EXP\{\esssup_{A\in \B_{d+1}
} 
| \PROB\{   (X_{k+s},Y_{k+s}) \in A| X_1, \epsilon_1, \dots, X_k, \epsilon_k\}
-
\PROB\{ (X_{k+s},Y_{k+s}) \in A\}|
\}
\\
&&
=
\sup_{k\in\N} 
\EXP\{\esssup_{A\in \B_{d+1}
} 
| \PROB\{   (X_{k+s},Y_{k+s}) \in A| X_1 \dots, X_k	\}
-
\PROB\{ (X_{k+s},Y_{k+s}) \in A\}|
\}
\end{eqnarray*}
where the last equality holds
because of the independence of $(X_t)_t$ and $(\epsilon_t)_t$, and
by using the symmetry of the $\beta$-mixing coefficient (cf., e.g.,
Remark 2.4 in Barrera and Gobet (2021)) we can conclude in the same
way that this in turn is bounded by
\begin{eqnarray*}
&&
\sup_{k\in\N} 
\EXP\{\esssup_{A\in \B_{d}
} 
| \PROB\{   X_{k+s} \in A| X_1, \dots, X_k\}
-
\PROB\{ X_{k+s} \in A\}|
\}.
\end{eqnarray*}
Examples of sequences of exponentially $\beta$-mixing $\R^d$--valued
random variables can be found, e.g., in Kurisu, Fukami and Koike (2024).
A simple example of a exponentially $\beta$-mixing sequence of identically
distributed real valued random variables is given by the $AR(1)$-model
\[
X_{t+1} = \frac{1}{\sqrt{2}} \cdot X_t +  \frac{1}{\sqrt{2}} \cdot \bar{\epsilon}_{t+1} \quad (t \in \N)
\]
where $X_1, \bar{\epsilon}_2, \bar{\epsilon}_3, \dots$ are independent
and identically standard normally distributed random variables.

    \section{Proofs}
\label{se4}

In the proof of Theorem \ref{th1} we will need results which help us to analyze the
optimization error of the estimate, the approximation error of the
estimate, and the generalization error of the estimate, which we
present before the proof in separate subsections.

\subsection{Neural network optimization}
\label{se4sub1}

Let $d, J_n \in \N$, and for
\[
\bw=(w_1, \dots, w_{J_n}) \in \R^{J_n}
\]
let
$
  f_{\bw}: \R^d \rightarrow \R
  $
  be a (deep) neural network with weight vector $\bw$ as defined by (\ref{se2eq1}) - (\ref{se2eq3}).
Further let
   \begin{equation}\label{se4eq1}
    f_{lin, \bw, \tilde{\bw}}(x)
    =
    f_{\bw}(x)+
    \sum_{j=1}^{J_n}
    \frac{\partial f_{\bw}(x)}{ \partial \bw^{(j)}}
      \cdot
      (\tilde{\bw}^{(j)} - \bw^{(j)})
   \end{equation}
be the linear Taylor polynomial of $f_{\tilde{\bw}}(x)$ around $\bw$.
  
\begin{lemma}
  \label{le1}
  Let $F_n$ be defined by (\ref{se2eq5}), $\bw^{(t+1)}$ by (\ref{se2eq6}) and $f_{lin, \bw, \tilde{\bw}}$ by (\ref{se4eq1}).
  Let $\bw^*\in\R^{J_n}$.
Assume
  \begin{equation}
    \label{le1eq1}
    F_n( \bw^{(s+1)}) \leq F_n(\bw^{(s)}) - \frac{\lambda}{2}
    \cdot \left\|
\nabla_\bw F_n(\bw^{(s)})
    \right\|^2
    \end{equation}
for $s=0,1, \dots, t_n-1$,
  \begin{equation}
    \label{le1eq2}
    \left|
f_{lin,\bw,\bw^*}(x)-f_{\bw^*}(x)
\right|
\leq
C_n \cdot \| \bw^* - \bw\|^2
    \end{equation}
  for all $x \in \{X_1, \dots, X_n \}$ and all $\bw \in \R^{J_n}$ with
  $
  \| \bw^* - \bw\|
  \leq
  \| \bw^* - \bw^{(0)}\|,
  $
  \begin{equation}
    \label{le1eq3}
    |Y_i| \leq \kappa_n \quad (i=1, \dots, n),
    \end{equation}
  \begin{equation}
    \label{le1eq4}
    |f_{\bw^*}(X_i)| \leq \kappa_n \quad (i=1, \dots, n)  
    \end{equation}
and
  \begin{equation}
    \label{le1eq5}
    C_n \cdot  \| \bw^* - \bw^{(0)}\|^2 \leq \kappa_n.
  \end{equation}
  Then
  \begin{eqnarray*}
    &&
    F_n (\bw^{(t_n)})
    \leq
    F_n( \bw^*)
    +
    \left(
5 \cdot \kappa_n \cdot C_n + \frac{1}{2 \cdot \lambda \cdot t_n}
    \right)
    \cdot
    \| \bw^* - \bw^{(0)}\|^2
    +
    \frac{F_n( \bw^{(0)})}{t_n}.
    \end{eqnarray*}
\end{lemma}

In the proof of Lemma \ref{le1} we will need the following auxiliary result.

\begin{lemma}
  \label{le2}
  Let $t \in \{ 1, \dots, t_n\}$.
  Set
  \[
  F_{n,lin,\bw}(\bw^*)
  =
\frac{1}{n} \sum_{i=1}^n |Y_i - f_{lin,\bw,\bw^*}(X_i)|^2,  
\]
and assume (\ref{le1eq1}) and
\begin{equation}
  \label{le2eq1}
F_n( \bw^{(s+1)}) \geq F_{n,lin,\bw^{(s)}}(\bw^*)
\end{equation}
for $s=0,1, \dots, t-1$. Then
\[
\| \bw^{(t)} - \bw^{*}\|
\leq
\| \bw^{(0)} - \bw^{*}\|.
\]
\end{lemma}

\noindent
    {\bf Proof.}
    Let $s \in \{0,1, \dots, t-1\}$ be arbitrary.
    We have
    \[
    F_{n,lin,\bw^{(s)}}(\bw^{(s)})=     F_{n}(\bw^{(s)})
    \quad \mbox{and} \quad
    \nabla_\bw F_{n,lin,\bw^{(s)}}(\bw^{(s)})=     \nabla_\bw F_{n}(\bw^{(s)})
    \]
    and (since $ \bw \mapsto  F_{n,lin,\bw^{(s)}}(\bw)$ is convex
    and differentiable)
    \[
    F_{n,lin,\bw^{(s)}}(\bw^{(s)})
    -
    F_{n,lin,\bw^{(s)}}(\bw^*)
    \leq
    \,
    <
    \nabla_\bw F_{n,lin,\bw^{(s)}}(\bw^{(s)}),
    \bw^{(s)} - \bw^*
    >,
    \]
    which implies
    \begin{eqnarray*}
    <
    \nabla_\bw F_{n,}(\bw^{(s)}),
    \bw^{(s)} - \bw^*
    >
    &=&
     <
    \nabla_\bw F_{n,lin,\bw^{(s)}}(\bw^{(s)}),
    \bw^{(s)} - \bw^*
    >
    \\
    &\geq&
     F_{n,lin,\bw^{(s)}}(\bw^{(s)})
    -
    F_{n,lin,\bw^{(s)}}(\bw^*)
    \\
    &=&
     F_{n}(\bw^{(s)})
    -
    F_{n,lin,\bw^{(s)}}(\bw^*).
    \end{eqnarray*}
    Hence
    \begin{eqnarray*}
      &&
\| \bw^{(s+1)} - \bw^{*}\|^2
\\
&&
=
\left\|
\bw^{(s)}
-
\lambda
\cdot \nabla_{\bw} F_n( \bw^{(s)} )
- \bw^{*}
\right\|^2
\\
&&
=
\left\|
\bw^{(s)}
- \bw^{*}
\right\|^2
-
2 \cdot \lambda \cdot
    <
    \nabla_\bw F_{n}(\bw^{(s)}),
    \bw^{(s)} - \bw^*
    >
+
\lambda^2
\cdot
\left\|
 \nabla_{\bw} F_n( \bw^{(s)} )
\right\|^2
\\
&&
\leq
\left\|
\bw^{(s)}
- \bw^{*}
\right\|^2
-
2 \cdot \lambda
\left(
     F_{n}(\bw^{(s)})
    -
    F_{n,lin,\bw^{(s)}}(\bw^*)
    \right)
    \\
    &&
    \quad
    + 2 \cdot \lambda \cdot (F_{n}(\bw^{(s)})
    -F_{n}(\bw^{(s+1)}))
\\
&&
=
\left\|
\bw^{(s)}
- \bw^{*}
\right\|^2
-
2 \cdot \lambda
\left(
     F_{n}(\bw^{(s+1)})
    -
    F_{n,lin,\bw^{(s)}}(\bw^*)
    \right)
    \\
    &&
    \leq
    \left\|
    \bw^{(s)}
- \bw^{*}
\right\|^2.
    \end{eqnarray*}
    Here we have used (\ref{le1eq1}) in the first inequality and (\ref{le2eq1})
    in the second inequality.
    \hfill $\Box$

    \noindent
        {\bf Proof of Lemma \ref{le1}.}
        Set
  \[
  F_{n,lin,\bw}(\bw^*)
  =
\frac{1}{n} \sum_{i=1}^n |Y_i - f_{lin,\bw,\bw^*}(X_i)|^2 . 
\]

In the {\it first step of the proof} we show
that for any
$\bw \in \R^{J_n}$ with
  $
  \| \bw^* - \bw\|
  \leq
  \| \bw^* - \bw^{(0)}\|,
  $
  we have
  \begin{equation}
    \label{ple1eq1}
    \left|
  F_{n,lin,\bw}(\bw^*)
-  F_{n}(\bw^*)
    \right|
    \leq
    5 \cdot \kappa_n \cdot C_n \cdot \| \bw^* - \bw \|^2.
    \end{equation}
  Using (\ref{le1eq2})--(\ref{le1eq5}) we get
  \begin{eqnarray*}
    &&
    \left|
  F_{n,lin,\bw}(\bw^*)
-  F_{n}(\bw^*)
    \right|
    \\
    &&
    =
        \left|
  \frac{1}{n} \sum_{i=1}^n |Y_i - f_{lin,\bw,\bw^*}(X_i)|^2
-  \frac{1}{n} \sum_{i=1}^n |Y_i - f_{\bw^*}(X_i)|^2
    \right|
    \\
    &&
    \leq
    \frac{1}{n} \sum_{i=1}^n
    | 2 \cdot Y_i - 2 \cdot f_{\bw^*}(X_i) +  f_{\bw^*}(X_i) -    f_{lin,\bw,\bw^*}(X_i)|
    \cdot | f_{lin,\bw,\bw^*}(X_i) - f_{\bw^*}(X_i)|
    \\
    &&
    \leq
    \frac{1}{n} \sum_{i=1}^n (2 \cdot \kappa_n + 2 \cdot \kappa_n + C_n \cdot \|\bw^*-\bw\|^2)
    \cdot C_n \cdot \|\bw^*-\bw\|^2
    \\
    &&
    \leq
    5 \cdot \kappa_n \cdot C_n \cdot \|\bw^*-\bw\|^2.
    \end{eqnarray*}

  In the {\it second step of the proof} we show that the assertion holds
  in case that we have for some $s \in \{0,1, \dots, t_n-1\}$
  \begin{equation}
    \label{ple1eq2}
F_n( \bw^{(s+1)}) < F_{n,lin,\bw^{(s)}}(\bw^*).
  \end{equation}

  So assume that (\ref{ple1eq2}) holds for some $s \in \{0,1, \dots, t_n-1\}$.
  By choosing $s$ minimal with this property we can assume
  \[
F_n( \bw^{(t+1)}) \geq F_{n,lin,\bw^{(t)}}(\bw^*)
  \]
  for all $t \in \{0,1, \dots, s-1\}$. By Lemma \ref{le2}
  we can conclude
  \[
\| \bw^{(s)} - \bw^{*}\|
\leq
\| \bw^{(0)} - \bw^{*}\|.
\]
Furthermore, we know by assumption (\ref{le1eq1})
\[
F_n( \bw^{(t_n)}) \leq F_n( \bw^{(s+1)}).
\]
By using the result of the first step of the proof we get
\begin{eqnarray*}
  F_n (\bw^{(t_n)})
  & \leq &
  F_n( \bw^{(s+1)})
  \\
  &
  <&
  F_{n,lin,\bw^{(s)}}(\bw^*)
  \\
  &
  =&
  F_n( \bw^*) +  F_{n,lin,\bw^{(s)}}(\bw^*) - F_n( \bw^*)
  \\
  &
  \leq&
  F_n(\bw^*) +     5 \cdot \kappa_n \cdot C_n \cdot \| \bw^* - \bw^{(s)} \|^2
  \\
  &
  \leq&
  F_n(\bw^*) +     5 \cdot \kappa_n \cdot C_n \cdot \| \bw^* - \bw^{(0)} \|^2
.
  \end{eqnarray*}

In the {\it third step of the proof} we show
the assertion in case that (\ref{ple1eq2}) does not hold for all
$s \in \{0,1, \dots, t_n-1\}$.

In this case we have
\[
F_n( \bw^{(s+1)}) \geq F_{n,lin,\bw^{(s)}}(\bw^*)
\]
for all $s \in \{0,1, \dots, t_n-1\}$, so by Lemma \ref{le2}
we know
\[
\| \bw^{(t)} - \bw^{*}\|
\leq
\| \bw^{(0)} - \bw^{*}\|
\]
for all $t \in \{0,1, \dots, t_n\}$.

Using (\ref{le1eq1}) and the result of the first step of the proof we
conclude
\begin{eqnarray*}
  &&
  F_n( \bw^{(t_n)}) - F_n (\bw^*)
  \\
  &&
  \leq
  \frac{1}{t_n} \sum_{t=0}^{t_n-1}  F_n( \bw^{(t)}) - F_n (\bw^*)
  \\
  &&
  =
  \frac{1}{t_n} \sum_{t=0}^{t_n-1}  \left( F_n( \bw^{(t)}) -
F_{n,lin, \bw^{(t)}} (\bw^*)
  \right)
  +
  \frac{1}{t_n} \sum_{t=0}^{t_n-1}  \left(
  F_{n,lin, \bw^{(t)}} (\bw^*)
    - F_n (\bw^*)
    \right)
    \\
    &&
    \leq
     \frac{1}{t_n} \sum_{t=0}^{t_n-1}  \left( F_n( \bw^{(t)}) -
F_{n,lin, \bw^{(t)}} (\bw^*)
  \right)
  +
  \frac{1}{t_n} \sum_{t=0}^{t_n-1}
  5 \cdot \kappa_n \cdot C_n \cdot \| \bw^* - \bw^{(t)} \|^2
    \\
    &&
    \leq
     \frac{1}{t_n} \sum_{t=0}^{t_n-1}  \left( F_n( \bw^{(t)}) -
F_{n,lin, \bw^{(t)}} (\bw^*)
  \right)
  +
  5 \cdot \kappa_n \cdot C_n \cdot \| \bw^* - \bw^{(0)} \|^2.
  \end{eqnarray*}
Since
    \[
    F_{n,lin,\bw^{(t)}}(\bw^{(t)})=     F_{n}(\bw^{(t)})
    \quad \mbox{and} \quad
    \nabla_\bw F_{n,lin,\bw^{(t)}}(\bw^{(t)})=     \nabla_\bw F_{n}(\bw^{(t)})
    \]
    and $ \bw \mapsto  F_{n,lin,\bw^{(t)}}(\bw)$ is convex
    and differentiable we get furthermore
    \begin{eqnarray*}
      &&
     \frac{1}{t_n} \sum_{t=0}^{t_n-1}  \left( F_n( \bw^{(t)}) -
F_{n,lin, \bw^{(t)}} (\bw^*)
  \right)
  \\
  &&
  =
       \frac{1}{t_n} \sum_{t=0}^{t_n-1}  \left( F_{n,lin, \bw^{(t)}}( \bw^{(t)}) -
F_{n,lin, \bw^{(t)}} (\bw^*)
  \right)
  \\
  &&
  \leq
  \frac{1}{t_n} \sum_{t=0}^{t_n-1}
  <
\nabla_\bw F_{n,lin, \bw^{(t)}}
( \bw^{(t)}), \bw^{(t)} -  \bw^{*}
>
\\
&&
=
  \frac{1}{t_n} \sum_{t=0}^{t_n-1}
  <
\nabla_\bw F_{n}
( \bw^{(t)}), \bw^{(t)} - \bw^{*}
>
\\
&&
=
\frac{1}{2 \cdot \lambda \cdot t_n} \sum_{t=0}^{t_n-1}
2 \cdot
  <
\lambda \cdot \nabla_\bw F_{n}
( \bw^{(t)}), \bw^{(t)} -  \bw^{*}
>
\\
&&
=
\frac{1}{2 \cdot \lambda \cdot t_n} \sum_{t=0}^{t_n-1}
\Bigg(
\| \bw^{(t)} -  \bw^{*} \|^2
-
\|
 \bw^{(t)} -  \bw^{*} -\lambda \cdot \nabla_\bw F_{n}
( \bw^{(t)})
 \|^2
 \\
 &&
 \hspace*{3cm}
 +
 \lambda^2 \cdot \|  \nabla_\bw F_{n}
( \bw^{(t)})\|^2
 \Bigg)
 \\
 &&
 =\frac{1}{2 \cdot \lambda \cdot t_n} \sum_{t=0}^{t_n-1}
\left(
\| \bw^{(t)} -  \bw^{*} \|^2
-
\|
 \bw^{(t+1)} - \bw^*
 \|^2
 \right)
 +
 \frac{\lambda}{2 \cdot t_n} \sum_{t=0}^{t_n-1} \|  \nabla_\bw F_{n}
( \bw^{(t)})\|^2
 \\
 &&
 \leq\frac{1}{2 \cdot \lambda \cdot t_n}
 \cdot
 \left(
\| \bw^{(0)} -  \bw^{*} \|^2
-
\|
 \bw^{(t_n)} - \bw^*
 \|^2
 \right)
 +
 \frac{1}{t_n} \sum_{t=0}^{t_n-1} ( F_n (\bw^{(t)}) - F_n(\bw^{(t+1)}))
 \\
 &&
 \leq
 \frac{1}{2 \cdot \lambda \cdot t_n}
 \cdot
 \| \bw^{(0)} -  \bw^{*} \|^2
 +
 \frac{
 F_n (\bw^{(0)})
 }{t_n},
      \end{eqnarray*}
      where the second to last inequality follows from (\ref{le1eq1}).
        \hfill $\Box$

Next we investigate when our topology of the deep neural network  satisfies the assumptions of Lemma \ref{le1}. The following localization lemma for gradient descent proven in Braun et al. (2024) helps with inequality (\ref{le1eq1}).

\begin{lemma}
  \label{le3}
    Let
  $F:\R^K \rightarrow \R_+$
  be a nonnegative differentiable function.
  Let
  $t \in \N$, $\bar L>0$, $\ba_0 \in \R^K$, choose
  \[
  0< \lambda \leq
\frac{1}{\bar L}
\]
and set
\[
\ba_{k+1}=\ba_k - \lambda \cdot (\nabla_{\ba} F)(\ba_k)
\quad
(k \in \{0,1, \dots, t-1\}).
\]
Assume
\begin{equation}
  \label{le3eq1}
  \left\|
 (\nabla_{\ba} F)(\ba)
  \right\|
  \leq
  \sqrt{
2 \cdot t \cdot\bar L \cdot \max\{ F(\ba_0),1 \}
    }
\end{equation}
for all $\ba \in \R^K$ with
$\| \ba - \ba_0\| \leq \sqrt{2 \cdot t \cdot \max\{ F(\ba_0),1 \} /\bar L}$,
and
\begin{equation}
  \label{le3eq2}
\left\|
(\nabla_{\ba} F)(\ba)
-
(\nabla_{\ba} F)(\bb)
  \right\|
  \leq
  \bar L \cdot \|\ba - \bb \|
\end{equation}
for all $\ba, \bb \in \R^K$ satisfying
\begin{equation}
  \label{le3eq3}
  \| \ba - \ba_0\| \leq \sqrt{8 \cdot \frac{t}{\bar L} \cdot \max\{ F(\ba_0),1 \}}
  \quad \mbox{and} \quad
  \| \bb - \ba_0\| \leq \sqrt{8 \cdot \frac{t}{\bar L} \cdot \max\{ F(\ba_0),1 \}}.
\end{equation}
Then we have
\[
\|\ba_k-\ba_0\| \leq
\sqrt{
2 \cdot \frac{k}{\bar L} \cdot (F(\ba_0)-F(\ba_k))
}
\quad
 \mbox{for all }
 k \in \{1, \dots,t\},
\]
\[
\sum_{k=0}^{s-1}
\| \ba_{k+1}-\ba_k \|^2
\leq
\frac{2}{\bar L}
 \cdot (F(\ba_0)-F(\ba_s))
\quad
 \mbox{for all }
 s \in \{1, \dots,t\}
 \]
 and
 \[
 F(\ba_k) \leq F(\ba_{k-1}) - \frac{\lambda}{2}
 \cdot
\left\| (\nabla_{\ba} F)(\ba_{k-1}) \right\|^2
 \quad
 \mbox{for all }
 k \in \{1, \dots,t\}.
 \]
\end{lemma}

\noindent
    {\bf Proof.} See Lemma A.1 in Braun et al. (2024).
    \hfill $\Box$
    
We use the following two lemmata from Kohler (2026) to verify assumptions (\ref{le3eq1}) and (\ref{le3eq2}) of Lemma \ref{le3}.
 \begin{lemma}
      \label{le3Ko24}
Let $\sigma: \R \rightarrow \R$ be the logistic squasher, let $L,r,K_n\in\N$,
let $f_{\bw}$ be defined by (\ref{se2eq1}) - (\ref{se2eq3}) and let
$F_n$ be defined by (\ref{se2eq5}).
Let $a \geq 1$,
$\gamma_n^* \geq 1$, $B_n \geq 1$, $\kappa_n \geq 1$  and assume
$X_i \in [-a,a]^d$, $|Y_i| \leq \kappa_n$ $(i=1, \dots, n)$,
\[
  |w_{1,1,k}^{(L)}| \leq \gamma_n^*\text{ for } k=1,\dots,K_n,
\]
\[
  |w_{k,i,j}^{(l)}| \leq B_n \text{ for }l = 1, \dots, L-1 \text{ and all } k,i,j
\]
and
\[
  K_n \cdot \gamma_n^* \geq \kappa_n.
\]

Then
\[
\| \nabla_\bw F_n(\bw) \| \leq \const[c1le3Ko24] \cdot
K_n^{3/2} \cdot (\gamma_n^*)^2 \cdot B_n^L 
  \]
    for some $\const[c1le3Ko24]=\const[c1le3Ko24](d,L,r,a)>0$.
      \end{lemma}

\noindent
    {\bf Proof.} See Lemma 3 in Kohler (2026).
    \hfill $\Box$

    \begin{lemma}
      \label{le5Ko24}
      Let $\sigma: \R \rightarrow \R$ be the logistic squasher, let $L,r,K_n\in\N$,
      let $f_{\bw}$ be defined by (\ref{se2eq1}) - (\ref{se2eq3}) and let
      $F_n$ be defined by (\ref{se2eq5}).
     Let $a \geq 1$,
     $\gamma_n^* \geq 1$, $B_n \geq 1$, $\kappa_n \geq 1$  and assume
     $X_i \in [-a,a]^d$, $|Y_i| \leq \kappa_n$ $(i=1, \dots, n)$,
     \begin{equation}
     	\label{le5Ko24eq1}
     	\max\{ |(\bw_1)_{1,1,k}^{(L)}|, |(\bw_2)_{1,1,k}^{(L)}|\} \leq \gamma_n^* \text { for } k=1,
     	\dots, K_n,
     \end{equation}
     \begin{equation}
     	\label{le5Ko24eq2}
     	\max\{|(\bw_1)_{k,i,j}^{(l)}|,|(\bw_2)_{k,i,j}^{(l)}|\} \leq B_n
     	\quad
        \mbox{for } l=1, \dots, L-1 \text{ and all } k,i,j
     \end{equation}
     and
  \[
K_n \cdot \gamma_n^* \geq \kappa_n.
  \]
     Then we have
     \begin{eqnarray*}
     	&&
     	\| (\nabla_\bw F_n)(\bw_1) - (\nabla_\bw F_n)(\bw_2) \|
     	\leq
     	\const[c1le5Ko24]   \cdot K_n^{3/2} \cdot B_n^{2L} \cdot (\gamma_n^*)^2 \cdot \|\bw_1-\bw_2\|
     \end{eqnarray*}
     for some $\const[c1le5Ko24]=\const[c1le5Ko24](d,L,r,a)>0$.
\end{lemma}

\noindent
    {\bf Proof.} See Lemma 5 in Kohler (2026).
    \hfill $\Box$

The following two lemmata consider inequality (\ref{le1eq2}) for the special topology of our networks.

\begin{lemma}
\label{le4} 
Let 
$f_\bw:\R^d \rightarrow \R$
be a deep neural network defined by (\ref{se2eq1}) - (\ref{se2eq3}) with weight vector
\[
\bw=(w_j)_{j=1, \dots, J_n} \in \R^{J_n}
\]
and a twice differentiable activation function,
and denote the linear Taylor polynomial  of $f_{\bw}$ around $\bw_0$ by
  \[
  f_{lin,\bw_0,\bw}(x)
  =
  f_{\bw_0}(x) + \sum_{j=1}^{J_n}
  \frac{\partial}{\partial w_j} f_{\bw_0}(x) \cdot
    (w_{j}-(\bw_{0})_j).
    \]
Then for every $x\in\R^d$ there exists $\xi \in [0,1]$ such that
\[
| f_{lin, \bw_0,\bw}(x) - f_\bw (x)|
\leq
\frac{1}{2}
\cdot
\| \bH( \bw_0 + \xi \cdot (\bw-\bw_0)) \|_2 \cdot \| \bw-\bw_0 \|^2,
\]
where 
\[
\bH (\bw)
=
\left(
\frac{\partial^2 f_\bw (x)}{
\partial w_i \partial w_j
}
\right)_{
1 \leq i,j \leq J_n
}
\]
is the Hessian matrix of $f_\bw (x)$ and
\[
\left\|
\bH (\bw)
\right\|
_2
=
\sup_{
\tilde{\bw} \in \R^{J_n}: \tilde{\bw} \neq 0
}
\frac{
\|  \bH (\bw) \tilde{\bw}\|
}{
\| \tilde{\bw} \|
}
\]
denotes its spectral norm.
\end{lemma}

\noindent
{\bf Proof.}
For $x\in\R^d$ and $s \in [0,1]$ define
\[
F(s)= f_{\bw_0 + s \cdot (\bw-\bw_0)}(x).
\]
Then the chain rule and the formula for Taylor polynomials of order 2 imply that for some
$\xi \in [0,1]$ we have
\begin{eqnarray*}
&&
|f_\bw (x)- f_{lin, \bw_0,\bw}(x) |
\\
&&
= | F(1)-F(0) - F^\prime(0) \cdot (1-0)|
\\
&&
=
\left|
\frac{1}{2}
\cdot
F^{\prime \prime}(\xi) \cdot (1-0)^2
\right|
\\
&&
=
\left|
\frac{1}{2} \cdot
(\bw-\bw_0)^T
\cdot
 \bH( \bw_0 + \xi \cdot (\bw-\bw_0)) 
\cdot (\bw-\bw_0)
\right|
\\
&&
\leq
\frac{1}{2} \cdot
\left\|
\bw-\bw_0
\right\|
\cdot
\left\|
 \bH( \bw_0 + \xi \cdot (\bw-\bw_0)) 
\cdot (\bw-\bw_0)
\right\|
\\
&&
\leq
\frac{1}{2} 
\cdot
\left\|
 \bH( \bw_0 + \xi \cdot (\bw-\bw_0)) 
\right\|_2
\cdot
\left\|
\bw-\bw_0
\right\|^2.
\end{eqnarray*}
\hfill $\Box$

\begin{lemma}
\label{le5}
Let $\sigma$ be the logistic squasher and define $f_\bw$ by (\ref{se2eq1}) - (\ref{se2eq3}).
Let $\const[c1le5]>0$ and assume
\[
|  w_{k,i,j}^{(l)} | \leq \const[c1le5]
\]
for all $l \in \{1, \dots, L \}$ and all $k,i,j$.
Let $R >0$. Then we have for all $x \in \R^d$ with $\|x\| \leq R$
\[
\| \bH (\bw) \|_2 \leq \const[c2le5]
\]
for some constant $\const[c2le5]=\const[c2le5](L,r,d,R,\const[c1le5])>0$ which does not depend on $K_n$.
\end{lemma}

\noindent
{\bf Proof.}
Since
\[
\frac{\partial^2 f_\bw (x)}{
\partial w_{k_1,i_1,j_1}^{(l_1)}
\partial w_{k_2,i_2,j_2}^{(l_2)}
}
=0
\]
holds whenever $k_1 \neq k_2$, the Hessian matrix is a block diagonal matrix
given by
\[
\left(
\begin{array}{cccc}
\bA_1 & 0 & \dots & 0 \\
0 & \bA_2 & \dots & 0 \\
\vdots & \vdots & \vdots & \vdots \\
0 & 0 & \dots & \bA_{K_n}
\end{array}
\right).
\]
If we split $\tilde{\bw}$ accordingly into
$\tilde{\bw}=(\tilde{\bw}_1, \dots, \tilde{\bw}_{K_n})^T$, then we have
\begin{eqnarray*}
&&
\| \bH(\bw) \cdot \tilde{\bw} \|^2
=
\left\|
\left(
\begin{array}{c}
\bA_1 \tilde{\bw}_1  \\
 \vdots \\
 \bA_{K_n} \tilde{\bw}_{K_n}
\end{array}
\right)
\right\|^2
=
\sum_{k=1}^{K_n}
\| \bA_k \tilde{\bw}_k \|^2
\leq
\sum_{k=1}^{K_n}
\| \bA_k  \|_2^2 \cdot \| \tilde{\bw}_k \|^2
\\
&&
\leq
\max \{
\| \bA_1  \|_2^2 , \dots, , \| \bA_{K_n}  \|_2^2 
\}
\cdot
\sum_{k=1}^{K_n}
 \| \tilde{\bw}_k \|^2
\\
&&
=
\max \{
\| \bA_1  \|_2^2 , \dots, , \| \bA_{K_n}  \|_2^2 
\}
\cdot
 \| \tilde{\bw} \|^2,
\end{eqnarray*}
which implies
\[
\| \bH (\bw) \|_2 \leq
\max \{
\| \bA_1  \|_2 , \dots, , \| \bA_{K_n}  \|_2
\}
.
\]
By construction, each matrix $\bA_k$ is a square matrix of size
\[
K=r+2+(L-2)\cdot r\cdot (r+1)+r\cdot (d+1),
\]
where all entries are bounded by a constant depending only on $L$, $r$, $d$, $\const[c1le5]$ and $\|x\|_\infty$.
It is easy to see that the spectral norm of a matrix is bounded by its Frobenius norm, i.e., that
\[
\|A_k\|_2=
\| (a_{i,j})_{1 \leq i,j \leq K} \|_2
\leq
\sqrt{
\sum_{i=1}^K \sum_{j=1}^K a_{i,j}^2
}
\]
holds,
which implies the assertion. \hfill $\Box$

    We summarize all our results concerning neural network optimization in the following theorem.
    
\begin{theorem}\label{th2}
Let $\sigma: \R \rightarrow \R$ be the logistic squasher, let $f_{\bw}$ be defined by (\ref{se2eq1}) - (\ref{se2eq3}), let $L,r,K_n\in\N$ and let
$F_n$ be defined by (\ref{se2eq5}).
Let $\const[th2newc1]>0$, $\const[th2newc2]\geq 1$, $\const[Bn]\geq 1$, $\kappa_n\geq 1$, $A\geq 1$ 
and assume $X_1,\dots X_n\in [-A,A]^d, \text{ }|Y_i|\leq \kappa_n \text{ }(i\in\{1,\dots ,n\} )$ and
\[
\const[th2newc2]\cdot K_n\geq \kappa_n \text{ and } \frac{\kappa_n}{n}\leq \const[th2newc1].
\]
Choose a starting vector $\bw^{(0)}$ which satisfies
\[
(\bw^{(0)})^{(L)}_{1,1,k}=0 \text{ and }|(\bw^{(0)})^{(l)}_{k,i,j}|\leq \const[Bn]
\]
for all $l \in \{1, \dots, L-1\}$ and all $i,j,k$,
and set
\[
\bw^{(t+1)} = \bw^{(t)} - \lambda_n \cdot \nabla_\bw F_n(\bw^{(t)})
\]
for $t=0, 1, \dots, t_n-1$. Set
\[
\lambda_n= \frac{\const[th2newc3]}{K_n^{3/2}  
\cdot \kappa_n}
\quad
\mbox{and}
\quad
t_n = \left \lceil
\const[th2newc4] \cdot
\frac{K_n^{3/2}  }{ \kappa_n} 
\right \rceil 
\]
for $\const[th2newc3],\const[th2newc4]>0$.
Let 
\begin{equation}\label{th2eq1}
\bw^*\in\left\{\bw\, : \,\|\bw-\bw^{(0)}\|\leq \const[c1th2]\cdot \frac{\sqrt{\kappa_n}}{\sqrt{n}}\right\}
\end{equation}
for $\const[c1th2]>0$
and assume (\ref{le1eq4}) holds.
Then we have for $\const[th2newc5],\const[th2newc6]>0$ and $n$ sufficiently large
\[
F_n(\bw^{(t_n)})\leq F_n(\bw^*)+\const[th2newc5]\cdot \kappa_n^2 
\cdot \|\bw^*-\bw^{(0)}\|^2+\const[th2newc6]\cdot \frac{\kappa^{3}_n}{K_n^{3/2}}.
\]

\end{theorem}

\noindent
    {\bf Proof.} The assertion follows from Lemma \ref{le1} with $\lambda=\lambda_n$ and we will now check that the assumptions (\ref{le1eq1}) - (\ref{le1eq5}) are fulfilled. Assumptions (\ref{le1eq3}) and (\ref{le1eq4}) are trivially fulfilled. We will now consider assumption (\ref{le1eq1}). This assumption follows directly from Lemma \ref{le3}.    
Noticing that
\[
\kappa_n\cdot \sqrt{t_n\cdot\lambda_n}\leq \const
\text{ and } F_n(\bw^{(0)})\leq \kappa_n^2,
\]
we can see that the assumptions of Lemma \ref{le3} are fulfilled if we use Lemma \ref{le3Ko24} and Lemma \ref{le5Ko24} with
\[
\tilde{\gamma}^* = \const[th2newc2] + \const\cdot \kappa_n \cdot \sqrt{t_n\cdot\lambda_n}\text{ and }
B_n = \const[Bn] + \const \cdot \kappa_n\cdot\sqrt{t_n\cdot\lambda_n}
\]
and
\[
\bar L=\frac{1}{\lambda_n}=\const\cdot K_n^{3/2} 
\cdot \kappa_n.
\]
Assumption (\ref{le1eq1}) is hence fulfilled.
Remember that $\left|\left(\bw^{(0)}\right)_{k,i,j}^{(l)}\right|$
is  bounded by a constant for $l>0$ and all $k,i,j$.
Using this and (\ref{th2eq1}) we can show that
$\left|\left(\bw+\xi \cdot (\bw^*-\bw )\right)_{k,i,j}^{(l)}\right|$
is bounded by a constant
for all $\bw\in \R^{J_n}$ with $\|\bw^*-\bw\|\leq \|\bw^*-\bw^{(0)}\|$ and all $\xi\in [0,1]$, $l\in\{1,\dots,L\}$ and $k,i,j$.
Assumption (\ref{le1eq2}) follows consequently from Lemma \ref{le4} and Lemma \ref{le5} with $C_n=\const[c2le5](L,r,d,A,\const[th2nc])$ for some $\const[th2nc] > 0$, where we use $R=A$.
Assumption (\ref{le1eq5}) follows from (\ref{th2eq1}) for large enough $n$. With Lemma \ref{le1} we conclude 
\begin{align*}
F_n(\bw^{(t_n)})
&\leq
F_n(\bw^*)
+ \left(\const\cdot \kappa_n+ \const\cdot \kappa_n^2 
\right)\cdot \|\bw^*-\bw^{(0)} \|^2
+ \const[th2lc]\cdot \frac{\kappa_n^3}{K_n^{3/2}}
\\
&\leq
F_n(\bw^*)
+\const \cdot \kappa_n^2 
\cdot \|\bw^*-\bw^{(0)} \|^2
+\const[th2lc]\cdot \frac{\kappa_n^3}{K_n^{3/2}}.
\end{align*}
    \hfill $\Box$
\subsection{Neural network approximation}
\label{se4sub2}
We use the following result from Kohler (2026) for bounding the approximation error.

\begin{theorem}
  \label{th3}
Let $d \in \N$,  $p=q+s$ where
$s \in (0,1]$ and $q \in \N_0$, $C>0$,
$A \geq 1$
and
$A_n, B_n, \gamma_n^* \geq 1$.
Let $\sigma$ be the logistic squasher.
For $L,r,K \in \N$
let $\F$ be the set of all networks $f_{\bw}$ defined by
\begin{equation}
f_\bw(x) = \sum_{j=1}^{r} w_{1,1,j}^{(L)} \cdot f_{j,1}^{(L)}(x)
\end{equation}
for some $w_{1,1,1}^{(L)}, \dots, w_{1,1,r}^{(L)} \in \mathbb{R}$, where
$f_{j,1}^{(L)}=f_{\bw,j,1}^{(L)}$ are recursively defined by
\begin{equation}
f_{k,i}^{(l)}(x) = 
f_{\bw,k,i}^{(l)}(x) = 
\sigma\left(\sum_{j=1}^{r} w_{k,i,j}^{(l-1)}\cdot f_{k,j}^{(l-1)}(x) + w_{k,i,0}^{(l-1)} \right)
\end{equation}
for some $w_{k,i,0}^{(l-1)}, \dots, w_{k,i, r}^{(l-1)} \in \mathbb{R}$
$(l=2, \dots, L)$
and
\begin{equation}
f_{k,i}^{(1)}(x) = 
f_{\bw,k,i}^{(1)}(x) = 
\sigma \left(\sum_{j=1}^d w_{k,i,j}^{(0)}\cdot x^{(j)} + w_{k,i,0}^{(0)} \right)
\end{equation}
for some $w_{k,i,0}^{(0)}, \dots, w_{k,i,d}^{(0)} \in \mathbb{R}$, where
the weight vector satisfies
\[
|w_{k,i,j}^{(0)}| \leq A_n, \quad
|w_{k,i,j}^{(l)}| \leq B_n \quad \mbox{and} \quad
|w_{k,i,j}^{(L)}| \leq \gamma_n^*
\]
for all $l \in \{1, \dots, L-1\}$ and all $k,i,j$, and set
\[
\HH = \left\{ \sum_{k=1}^{K^{d}} f_k \quad : \quad f_k \in \F \quad (k=1, \dots, K^{d})
\right\}.
\]
Let $L,r \in \N$ with
\[
L \geq \lceil \log_2(q+d) \rceil
\quad
\mbox{and}
\quad
r \geq 4 \cdot (p+d)^2,
\]
and set
\[
 A_n =A \cdot K \cdot \log K, \quad B_n=\const
\quad \mbox{and} \quad 
\gamma_n^*=\const \cdot K^{q+d}.
\]
Assume $K \geq \const[c1th3]$ for $\const[c1th3]>0$ sufficiently large.
Then there exists for any $(p,C)$--smooth $f:\R^d \rightarrow \R$
a neural network $h \in \HH$ such that
\[
\sup_{x \in [-A,A)^d} |f(x)-h(x)|
\leq
\frac{\const}{K^{p}}.
\]
\end{theorem}

\noindent
    {\bf Proof.} See Theorem 3 in Kohler (2026).
    \hfill $\Box$

\subsection{Neural network generalization}
\label{se4sub3}

Our next lemma is our main tool to analyze the generalization error
for exponentially $\beta$-mixing data.

\begin{lemma}\label{lea3}
  Assume (A4) - (A5),
let $n \in \N$ with $n>1$,
  and let $\F_n$ be a set of functions $f:\R^d\to \R$. Then we have
\begin{align*}
&\EXP\left\{
\sup_{f\in\F_n} \EXP \left\{ |(T_{\kappa_n}f)(X)-T_{\kappa_n} Y|^2 \right\}
-
\EXP \{ |m_{\kappa_n}(X)- T_{\kappa_n} Y|^2\}
\vphantom{
\hspace*{2cm}
-
2 \cdot \frac{1}{n} \sum_{i=1}^n
\left(
|m_n(X_i)-T_{\kappa_n} Y_i|^2
-
|m_{\kappa_n}(X_i)- T_{\kappa_n} Y_i|^2
\right)
}
\right.
\\
&
\left.
\vphantom{
\sup_{f\in\F_n} \EXP \left\{ |T_{\kappa_n}f(X)-T_{\kappa_n} Y|^2 \right\}
-
\EXP \{ |m_{\kappa_n}(X)- T_{\kappa_n} Y|^2\}
}
\hspace*{2cm}
-
2 \cdot \frac{1}{n} \sum_{i=1}^n
\left(
|(T_{\kappa_n}f)(X_i)-T_{\kappa_n} Y_i|^2
-
|m_{\kappa_n}(X_i)- T_{\kappa_n} Y_i|^2
\right)
\right\}
\\
&\quad\leq
4\cdot\kappa_n^2\cdot \frac{n}{(\log n)^2}\cdot \beta_{\lceil (\log
  n)^2\rceil}
((X_1,Y_1), (X_2,Y_2), \dots, )
+\epsilon
\\
&\quad\quad
+\int_\epsilon^\infty 
14\cdot \sup_{x_1^n\in \supp(X)^n}\mathcal N_1\left(\frac{u}{80\cdot \kappa_n},\F_n,x_1^n\right)\cdot\exp\left(-\const[plea3c1]\cdot\frac{u\cdot n}{\kappa_n^2\cdot (\log n)^2}\right) \, du
\end{align*}
for every $\epsilon>0$ and some $\const[plea3c1]>0$.
\end{lemma}
We will use the following auxiliary result to prove Lemma \ref{lea3}, which
allows us to reduce our problem to the case of independent data and
controls the probability of the resulting error with $\beta$-mixing coefficients.
\begin{lemma}
\label{lea2}
Let $Z_0$, $Z_1$, \dots, $Z_n$ be identically distributed
$\R^d$--valued random variables defined on a probability space
$(\Omega,\A,\PROB)$ such that $Z_0$ is independent from
$Z_1$, \dots, $Z_n$. Then there exists a probability space
$(\bar{\Omega},\bar{\A},\bar{\PROB})$ and random variables
$\bar{ Z}_0$, $\bar{ Z}_1$, \dots, $\bar{ Z}_n$,
$\bar{ Z}_1^*$, \dots, $\bar{ Z}_n^*$ defined on this probability
space
such that
\begin{equation}
\label{lea2eq1}
\PROB_{(Z_0,Z_1, \dots, Z_n)}
=
\bar{\PROB}_{(\bar{ Z}_0, \bar{ Z}_1, \dots, \bar{ Z}_n)},
\end{equation}
\begin{equation}
\label{lea2eq2}
\bar{ Z}_0, \bar{ Z}_1^*, \dots, \bar{ Z}_n^* \quad \mbox{are } i.i.d.
\end{equation}
and
\begin{eqnarray}
\label{lea2eq3}
&&
\bar{\PROB}\{ \exists k \in \{1,\dots, n\}: \bar{Z}_k \neq \bar{Z}^*_k
\}
\nonumber \\
&&
\leq
(n-1) \cdot \max_{s \in \{2, \dots, n\}}
\EXP \left\{
\sup_{C \in \C}
\left|
\PROB\{ Z_s \in C | Z_1, \dots, Z_{s-1} \} - \PROB\{Z_s \in C\}
\right| \right\},
\end{eqnarray}
where $\C$ is some countable subset of $\B_d$.
\end{lemma}
Lemma \ref{lea2} is
closely related to classical coupling results (see Berbee (1979),
Corollary 4.2.4 and Doukhan (1994), Theorem 1, Section 1.1) and its
extensions (c.f. Barrera and Gobet (2021), Lemma 2.10), and follows
from these well--known results. For the sake of completeness, we
provide nevertheless a complete and self-contained proof of this
result in the appendix.

\noindent
{\bf Proof of Lemma \ref{lea3}.}
Set $N_n = \lceil (\log n)^2\rceil$ and $I_{n,k}=\{k,k+N_n, k+2\cdot N_n,\dots\}\cap\{1,\dots,n\}$ for $k\in\{1,\dots,N_n\}$.
For each $n\in\N$ we have
\[
\bigcup_{k=1}^{N_n} I_{n,k} =\{1,\dots,n\}\text{ and } I_{n,k_1}\cap
I_{n,k_2}=\emptyset \text{ for all } k_1,k_2\in\{1,\dots,N_n\}, k_1 \neq k_2.
\]
Notice that 
$$|I_{n,k}|
\in
\left\{\left\lfloor\frac{n}{N_n}\right\rfloor,\left\lceil \frac{n}{N_n}\right\rceil\right\}
=
\left\{\left\lfloor\frac{n}{\lceil (\log n)^2\rceil}\right\rfloor,\left\lceil \frac{n}{\lceil (\log n)^2\rceil}\right\rceil\right\}
$$
and hence
\begin{equation}\label{plea3eq0}
|I_{n,k}|\leq \frac{n}{(\log n)^2}+1.
\end{equation}
We get with the convexity of the supremum
\begin{align*}
&
\EXP\left\{\sup_{f\in\F_n} \left(\EXP \left\{ |(T_{\kappa_n}f)(X)-T_{\kappa_n} Y|^2 \right\}
-
\EXP \{ |m_{\kappa_n}(X)- T_{\kappa_n} Y|^2\}
\vphantom{
\hspace*{2cm}
-
2 \cdot \frac{1}{n} \sum_{i=1}^n
\left(
|m_n(X_i)-T_{\kappa_n} Y_i|^2
-
|m_{\kappa_n}(X_i)- T_{\kappa_n} Y_i|^2
\right)
}
\right.
\right.
\\
&
\left.
\left.
\vphantom{
\sup_{f\in\F_n} \EXP \left\{ |T_{\kappa_n}f(X)-T_{\kappa_n} Y|^2 \right\}
-
\EXP \{ |m_{\kappa_n}(X)- T_{\kappa_n} Y|^2\}
}
\hspace*{2cm}
-
2 \cdot \frac{1}{n} \sum_{i=1}^{n}
\left(
|(T_{\kappa_n}f)(X_i)-T_{\kappa_n} Y_i|^2
-
|m_{\kappa_n}(X_i)- T_{\kappa_n} Y_i|^2
\right)
\right)
\right\}
\\
&\quad\leq
\frac{1}{n}\sum_{k=1}^{N_n}|I_{n,k}|
\EXP\left\{\sup_{f\in\F_n} \left(\EXP \left\{ |(T_{\kappa_n}f)(X)-T_{\kappa_n} Y|^2 \right\}
-
\EXP \{ |m_{\kappa_n}(X)- T_{\kappa_n} Y|^2\}
\vphantom{
\hspace*{2cm}
-
2 \cdot \frac{1}{|I_{n,k}|} \sum_{i\in I_{n,k}}
\left(
|m_n(X_i)-T_{\kappa_n} Y_i|^2
-
|m_{\kappa_n}(X_i)- T_{\kappa_n} Y_i|^2
\right)
}
\right.
\right.
\\
&
\left.
\left.
\vphantom{
\sup_{f\in\F_n} \EXP \left\{ |T_{\kappa_n}f(X)-T_{\kappa_n} Y|^2 \right\}
-
\EXP \{ |m_{\kappa_n}(X)- T_{\kappa_n} Y|^2\}
}
\hspace*{2cm}
-
2 \cdot \frac{1}{|I_{n,k}|} \sum_{i\in I_{n,k}}
\left(
|(T_{\kappa_n}f)(X_i)-T_{\kappa_n} Y_i|^2
-
|m_{\kappa_n}(X_i)- T_{\kappa_n} Y_i|^2
\right)
\right)
\right\}
\\
&\quad\leq
\max_{k=1,\dots,N_n}
\EXP\left\{
\sup_{f\in\F_n} \left(\EXP \left\{ |(T_{\kappa_n}f)(X)-T_{\kappa_n} Y|^2 \right\}
-
\EXP \{ |m_{\kappa_n}(X)- T_{\kappa_n} Y|^2\}
\vphantom{
\hspace*{2cm}
-
2 \cdot \frac{1}{|I_{n,k}|} \sum_{i\in I_{n,k}}
\left(
|m_n(X_i)-T_{\kappa_n} Y_i|^2
-
|m_{\kappa_n}(X_i)- T_{\kappa_n} Y_i|^2
\right)
}
\right.
\right.
\\
&
\left.
\left.
\vphantom{
\sup_{f\in\F_n} \EXP \left\{ |T_{\kappa_n}f(X)-T_{\kappa_n} Y|^2 \right\}
-
\EXP \{ |m_{\kappa_n}(X)- T_{\kappa_n} Y|^2\}
}
\hspace*{2cm}
-
2 \cdot \frac{1}{|I_{n,k}|} \sum_{i\in I_{n,k}}
\left(
|(T_{\kappa_n}f)(X_i)-T_{\kappa_n} Y_i|^2
-
|m_{\kappa_n}(X_i)- T_{\kappa_n} Y_i|^2
\right)
\right)
\right\}
\\
&\quad =:
\max_{k=1,\dots, N_n} \EXP\{g_{n,k}((X,Y),((X_i,Y_i))_{i\in I_{n,k}})\}.
\end{align*}
Here we used that $\sum_{k=1}^{N_n} |I_{n,k}|=n$.

Now we use Lemma \ref{lea2} with
\[
n=|I_{n,k}|,\text{ }Z_0=(X,Y)
\]
and
\[
(Z_i)_{i=1}^{|I_{n,k}|}
=((X_{k+(i-1) \cdot N_n},Y_{k+(i-1)\cdot N_n}))_{i=1}^{|I_{n,k}|}(=((X_j,Y_j))_{j\in I_{n,k}})
\]
for each $k\in\{1,\dots,N_n\}$.
It follows that for each $k\in\{1,\dots,N_n\}$ there exists a probability space $(\Omega^{(k)},\A^{(k)},\PROB^{(k)})$ and random variables $(X^{(k)},Y^{(k)})$, $((X_j^{(k)},Y_j^{(k)}))_{j\in I_{n,k}}$ and $((X_j^{*(k)},Y_j^{*(k)}))_{j\in I_{n,k}}$ defined on this probability space such that
\begin{equation}\label{plea3eq1}
\PROB_{((X,Y),((X_j,Y_j))_{j\in I_{n,k}})}
=
\PROB^{(k)}_{((X^{(k)},Y^{(k)}),((X_j^{(k)},Y_j^{(k)}))_{j\in I_{n,k}})},
\end{equation}
\begin{equation}\label{plea3eq2}
(X^{(k)},Y^{(k)}),((X_j^{*(k)},Y_j^{*(k)}))_{j\in I_{n,k}}\text{ are i.i.d.} 
\end{equation}
and
\begin{align*}
&\PROB^{(k)}\{\exists j\in I_{n,k} :(X_j^{(k)},Y_j^{(k)})\neq(X_j^{*(k)},Y_j^{*(k)})\}
\\
&\quad \leq
(|I_{n,k}|-1)\max_{1\in \{1,\dots, |I_{n,k}|-1\}}
\EXP\Bigg\{\sup_{C\in \C}\big|\PROB\{(X_{k+s\cdot N_n},Y_{k+s \cdot N_n})\in C|
\\
&\hspace{4em}
(X_k,Y_k),(X_{k+ N_n},Y_{k+ N_n}),\dots,(X_{k+(s-1)\cdot N_n},Y_{k+(s-1)\cdot N_n})\}
\\
&\hspace{12em}
-\PROB\{(X_{k+s \cdot N_n},Y_{k+s\cdot N_n})\in C\}\big| \Bigg\}
\\
&\quad =
(|I_{n,k}|-1)\max_{i\in I_{n,k}\setminus\{k\}} 
\\
&\hspace{4em}
\EXP\left\{ \sup_{C\in \C}|\PROB\{(X_i,Y_i)\in C|(X_j,Y_j) \, : \, j \in
  I_{n,k}, j<i\}-\PROB\{(X_i,Y_i)\in C\}| \right\}
\end{align*}
for some countable set $\C\subset B_{d+1}$.

We have for
$i \in I_{n,k} \setminus \{k\}$ 
and each $C\in \C$ by the tower property and Jensen's inequality
\begin{align*}
&|\PROB\{(X_i,Y_i)\in C|(X_j,Y_j) \, : \, j\in I_{n,k}, j<i\}-\PROB\{(X_i,Y_i)\in C\}|
\\
&\quad =
|\EXP\{1_{\{(X_i,Y_i)\in C\}}-\PROB\{(X_i,Y_i)\in C\}|(X_j,Y_j) \, : \, j\in I_{n,k}, j<i\}|
\\
&\quad =
|\EXP\{\EXP\{1_{\{(X_i,Y_i)\in C\}}-\PROB\{(X_i,Y_i)\in C\}|(X_m,Y_m)\, : \, m=1,\dots,i-N_n\}
\\
&\hspace{4em}
|(X_j,Y_j) \, : \, j\in I_{n,k}, j<i\}|
\\
&\quad \leq
\EXP\{|\EXP\{1_{\{(X_i,Y_i)\in C\}}-\PROB\{(X_i,Y_i)\in C\}|(X_m,Y_m)\, : \, m=1,\dots,i-N_n\}|
\\
&\hspace{4em}
|(X_j,Y_j) \, : \, j\in I_{n,k}, j<i\}
\\
&\quad =
\EXP\{
|\PROB\{(X_i,Y_i)\in C |(X_m,Y_m)\, : \, m=1,\dots,i-N_n\}-\PROB\{(X_i,Y_i)\in C\}|
\\
&\hspace{4em}
|(X_j,Y_j) \, : \, j\in I_{n,k}, j<i\}
\\
&\quad\leq
\EXP\{
\esssup_{A\in\B_{d+1}}
|\PROB\{(X_i,Y_i)\in A |(X_m,Y_m)\, : \, m=1,\dots,i-N_n\}-\PROB\{(X_i,Y_i)\in A\}|
\\
&\hspace{4em}
|(X_j,Y_j) \, : \, j\in I_{n,k}, j<i\}
\end{align*}
almost surely. 
This implies
\begin{eqnarray*}
  &&
\EXP\left\{ \sup_{C\in \C}|\PROB\{(X_i,Y_i)\in C|(X_j,Y_j) \, : \, j \in
  I_{n,k}, j<i\}-\PROB\{(X_i,Y_i)\in C\}| \right\}
  \\
  &&
  \leq
\EXP\{
\esssup_{A\in\B_{d+1}}
|\PROB\{(X_i,Y_i)\in A |(X_m,Y_m)\, : \, m=1,\dots,i-N_n\}-\PROB\{(X_i,Y_i)\in A\}|
\}  
  \end{eqnarray*}
and hence
\begin{align}
&\PROB^{(k)}(\exists j\in I_{n,k} :(X_j^{(k)},Y_j^{(k)})\neq(X_j^{*(k)},Y_j^{*(k)}))
\nonumber\\
&\quad\leq
(|I_{n,k}|-1)\cdot\beta_{N_n}((X_1,Y_1),(X_2,Y_2),\dots).\label{plea3eq3}
\end{align}

Denote with $\EXP^{(k)}$ the expectation with respect to $\PROB^{(k)}$.
We have by (\ref{plea3eq1})
\begin{align*}
&\EXP\{g_{n,k}((X,Y),((X_i,Y_i))_{i\in I_{n,k}})\}
=
\EXP^{(k)}\{g_{n,k}((X^{(k)},Y^{(k)}),((X_i^{(k)},Y_i^{(k)}))_{i\in I_{n,k}})\}
\\
&\quad=
\EXP^{(k)}\{g_{n,k}((X^{(k)},Y^{(k)}),((X_i^{(k)},Y_i^{(k)}))_{i\in I_{n,k}})\cdot 1_{\{\exists j\in I_{n,k}:(X_j^{(k)},Y_j^{(k)})\neq(X_j^{*(k)},Y_j^{*(k)})\}}\} \\
&\quad +
\EXP^{(k)}\{g_{n,k}((X^{(k)},Y^{(k)}),((X_i^{(k)},Y_i^{(k)}))_{i\in I_{n,k}})\cdot 1_{\{\forall j\in I_{n,k}:(X_j^{(k)},Y_j^{(k)})=(X_j^{*(k)},Y_j^{*(k)})\}}\} \\
&=: \EXP^{(k)}\{S_{1,n,k}\}+\EXP^{(k)}\{S_{2,n,k}\}.
\end{align*}
Since $g_{n,k}((x,y),((x_i,y_i))_{i\in I_{n,k}})\leq 4\cdot \kappa_n^2$, we get by (\ref{plea3eq0}) and (\ref{plea3eq3})
\begin{align*}
\EXP^{(k)}\{S_{1,n,k}\}
&\leq
4\cdot \kappa_n^2\cdot \PROB^{(k)}(\exists j\in I_{n,k}:(X_j^{(k)},Y_j^{(k)})\neq(X_j^{*(k)},Y_j^{*(k)}))
\\
&\leq
4\cdot\kappa_n^2\cdot (|I_{n,k}|-1)\cdot \beta_{N_n}((X_1,Y_1),(X_2,Y_2),\dots)
\\
&\leq
4\cdot\kappa_n^2\cdot \frac{n}{(\log n)^2}\cdot \beta_{N_n}((X_1,Y_1),(X_2,Y_2),\dots)
.
\end{align*}
We also have
\begin{align*}
&\EXP^{(k)}\{S_{2,n,k}\}
\leq
\EXP^{(k)}\{(g_{n,k}((X^{(k)},Y^{(k)}),((X_i^{*(k)},Y_i^{*(k)}))_{i\in I_{n,k}}))_+\}
\end{align*}
(where $z_+ = \max\{z, 0\}$ for $z \in \R$),
and the $(X_i^{*(k)},Y_i^{*(k)})$ $(i\in I_{n,k})$ are identically distributed for each $k\in\{1,\dots,N_n\}$.
This implies, that there exists $K_n\in\left\{\left\lfloor\frac{n}{\lfloor (\log n)^2\rfloor}\right\rfloor,\left\lceil \frac{n}{\lfloor (\log n)^2\rfloor}\right\rceil\right\}$
and random variables
\[
(\bar X,\bar Y),(\bar{X}_1,\bar Y_1),(\bar{X}_2,\bar Y_2),\dots, (\bar{X}_{K_n},\bar Y_{K_n})
\]
on a probability space $(\bar \Omega, \bar \A, \bar \PROB)$ that are i.i.d and have the same distribution as $(X,Y)$ such that
\begin{align*}
&\max_{k\in\{1,\dots, N_n\}}\EXP^{(k)}\{S_{2,n,k}\} =  \EXP^{(k^*)}\{S_{2,n,k}\}\\
&\quad\leq 
  \bar\EXP\left\{
  \Bigg(
\sup_{f\in\F_n} \left(\bar\EXP \left\{ |(T_{\kappa_n}f)(\bar X)-T_{\kappa_n} \bar Y|^2 \right\}
-
\bar\EXP \{ |m_{\kappa_n}(\bar X)- T_{\kappa_n} \bar Y|^2\}
\vphantom{
\hspace*{2cm}
-
2 \cdot \frac{1}{|I_{n,k}|} \sum_{i\in I_{n,k}}
\left(
|m_n(X_i)-T_{\kappa_n} Y_i|^2
-
|m_{\kappa_n}(X_i)- T_{\kappa_n} Y_i|^2
\right)
}
\right.
\right.
\\
&
\left.
\left.
\vphantom{
\sup_{f\in\F_n} \EXP \left\{ |T_{\kappa_n}f(X)-T_{\kappa_n} Y|^2 \right\}
-
\EXP \{ |m_{\kappa_n}(X)- T_{\kappa_n} Y|^2\}
}
\hspace*{2cm}
-
2 \cdot \frac{1}{K_n} \sum_{i=1}^{K_n}
\left(
|(T_{\kappa_n}f)(\bar X_i)-T_{\kappa_n} \bar Y_i|^2
-
|m_{\kappa_n}(\bar X_i)- T_{\kappa_n} \bar Y_i|^2
\right)
\right)
\Bigg)_+
\right\}
\\
&=: \bar\EXP\{S_{3,n}\}
,
\end{align*}
e.g.,
\[
(\bar X, \bar Y)=(X,Y),\{(\bar X_i,\bar Y_i)\}_{i=1}^{K_n}=\{(X_j^*,Y_j^*)\}_{j\in I_{n,k^*}},
\]
and
\[
(\bar \Omega,\bar \A, \bar \PROB)=(\Omega^{(k^*)},\A^{(k^*)},\PROB^{(k^*)}).
\]
From here we proceed as in the proof of Theorem 1 in Kohler (2026) to get the assertion.
We have for any $u>0$
\begin{align*}
&\bar\PROB(S_{3,n}>u) \\
&\quad\leq
\bar\PROB\left(
\exists f\in \F_n : \bar\EXP \left\{ |(T_{\kappa_n}f)(\bar X)-T_{\kappa_n} \bar Y|^2 \right\}
-
\bar\EXP \{ |m_{\kappa_n}(\bar X)- T_{\kappa_n} \bar Y|^2\}
\vphantom{
\hspace*{2cm}
-
2 \cdot \frac{1}{K_n} \sum_{i=1}^{K_n}
\left(
|(T_{\kappa_n}f)(\bar X_i^*)-T_{\kappa_n} \bar Y_i^*|^2
-
|m_{\kappa_n}(\bar X_i^*)- T_{\kappa_n} \bar Y_i^*|^2
\right)
>u
}
\right.
\\
&
\left.
\vphantom{
\exists f\in \F_n : \EXP \left\{ |(T_{\kappa_n}f)(\bar X)-T_{\kappa_n} \bar Y|^2 \right\}
-
\EXP \{ |m_{\kappa_n}(\bar X)- T_{\kappa_n} \bar Y|^2\}
}
\hspace*{2cm}
-
2 \cdot \frac{1}{K_n} \sum_{i=1}^{K_n}
\left(
|(T_{\kappa_n}f)(\bar X_i)-T_{\kappa_n} \bar Y_i|^2
-
|m_{\kappa_n}(\bar X_i)- T_{\kappa_n} \bar Y_i|^2
\right)
>u
\right)
\\
&\quad\leq
\bar\PROB\left(
\exists f\in \F_n : \bar\EXP \left\{ \left|\frac{(T_{\kappa_n}f)(\bar X)}{\kappa_n}-\frac{T_{\kappa_n} \bar Y}{\kappa_n}\right|^2 \right\}
-
\bar\EXP \left\{ \left|\frac{m_{\kappa_n}(\bar X)}{\kappa_n}- \frac{T_{\kappa_n} \bar Y}{\kappa_n}\right|^2\right\}
\vphantom{
\hspace*{3cm}
-
2 \cdot \frac{1}{K_n} \sum_{i=1}^{K_n}
\left(
|\frac{(T_{\kappa_n}f)(\bar X_i^*)}{\kappa_n}-\frac{T_{\kappa_n} \bar Y_i^*}{\kappa_n}|^2
-
|\frac{m_{\kappa_n}(\bar X_i^*)}{\kappa_n}- \frac{T_{\kappa_n} \bar Y_i^*}{\kappa_n}|^2
\right)
>u
}
\right.
\\
&
\left.
\vphantom{
\exists f\in \F_n : \EXP \left\{ |(T_{\kappa_n}f)(\bar X)-T_{\kappa_n} \bar Y|^2 \right\}
-
\EXP \{ |m_{\kappa_n}(\bar X)- T_{\kappa_n} \bar Y|^2\}
}
\hspace*{2cm}
-
 \frac{1}{K_n} \sum_{i=1}^{K_n}
\left(
\left|\frac{(T_{\kappa_n}f)(\bar X_i)}{\kappa_n}-\frac{T_{\kappa_n} \bar Y_i}{\kappa_n}\right|^2
-
\left|\frac{m_{\kappa_n}(\bar X_i)}{\kappa_n}- \frac{T_{\kappa_n} \bar Y_i}{\kappa_n}\right|^2
\right)
\right.
\\
&\hspace{1cm}\left.
\vphantom{
\hspace*{2cm}
-
2 \cdot \frac{1}{K_n} \sum_{i=1}^{K_n}
\left(
|\frac{(T_{\kappa_n}f)(\bar X_i)}{\kappa_n}-\frac{T_{\kappa_n} \bar Y_i}{\kappa_n}|^2
-
|\frac{m_{\kappa_n}(\bar X_i)}{\kappa_n}- \frac{T_{\kappa_n} \bar Y_i}{\kappa_n}|^2
\right)
>u
}
>\frac{1}{2} \cdot \left(
\frac{u}{\kappa_n^2}
+ \bar\EXP \left\{ \left|\frac{(T_{\kappa_n}f)(\bar X)}{\kappa_n}-\frac{T_{\kappa_n} \bar Y}{\kappa_n}\right|^2 \right\}
-
\bar\EXP \left\{ \left|\frac{m_{\kappa_n}(\bar X)}{\kappa_n}- \frac{T_{\kappa_n} \bar Y}{\kappa_n}\right|^2\right\}
\right)
\right).
\end{align*}
Theorem 11.4 in Gy\"orfi et al. (2002) allows us to conclude
\begin{align*}
&\bar\PROB(S_{3,n}>u)\\
&\quad\leq
14\cdot \sup_{x_1^n\in \supp(X)^n}\mathcal N_1\left(\frac{u}{80\cdot \kappa_n^2},\left\{\frac{1}{\kappa_n}\cdot f : f\in\F_n\right\},x_1^n\right)\cdot\exp\left(-\frac{u\cdot K_n}{5136\cdot\kappa_n^2}\right)\\
&\quad\leq
14\cdot \sup_{x_1^n\in \supp(X)^n}\mathcal N_1\left(\frac{u}{80\cdot \kappa_n},\F_n,x_1^n\right)\cdot\exp\left(-\const[plea3c1]\cdot\frac{u\cdot n}{\kappa_n^2\cdot (\log n)^2}\right).
\end{align*}
Note that we take $\sup_{x_1^n\in \supp(X)^n}$ instead of $\sup_{x_1^n\in (\R^d)^n}$, which can be concluded from the proof of Theorem 11.4.  Finally we can conclude for any $\epsilon>0$
\begin{align*}
\EXP\{S_{3,n}\} 
&\leq
\int_0^{\epsilon} \PROB(S_{3,n}>u)\, du
+
\int_\epsilon^{\infty} \PROB(S_{3,n}>u)\, du
\\
&\leq
\epsilon
+
\int_\epsilon^\infty 
14\cdot \sup_{x_1^n\in \supp(X)}\mathcal N_1\left(\frac{u}{80\cdot \kappa_n},\F_n,x_1^n\right)\cdot\exp\left(-\const[plea3c1]\cdot\frac{u\cdot n}{\kappa_n^2\cdot (\log n)^2}\right) \, du.
\end{align*}
     \hfill $\Box$

To bound the covering number in Lemma \ref{lea3} we proceed similarly as in Kohler (2026) and extend Lemma 12 from Kohler (2026) to our setting.

\begin{lemma}
  \label{le2ko24}
  Let $\M$ be a $d^*$-dimensional Lipschitz-manifold.
  Let $k \geq 3$, $\kappa\geq 1$ and let 
  $A,B,C \geq 1$.
  Let $\sigma:\R \rightarrow \R$ be $k$-times differentiable
  such that all derivatives up to order $k$ are bounded on $\R$.
  Let $L,r,K_n \in \N$ and
  let $\F$
  be the set of all functions $f_{\bw}$ defined by
\begin{equation}
f_\bw(x) = \sum_{j=1}^{K_n} w_{1,1,j}^{(L)} \cdot f_{j,1}^{(L)}(x)
\end{equation}
for some $w_{1,1,1}^{(L)}, \dots, w_{1,1,K_n}^{(L)} \in \mathbb{R}$, where
$f_{j,1}^{(L)}=f_{\bw,j,1}^{(L)}$ are recursively defined by
\begin{equation}
f_{k,i}^{(l)}(x) = 
f_{\bw,k,i}^{(l)}(x) = 
\sigma\left(\sum_{j=1}^{r} w_{k,i,j}^{(l-1)}\cdot f_{k,j}^{(l-1)}(x) + w_{k,i,0}^{(l-1)} \right)
\end{equation}
for some $w_{k,i,0}^{(l-1)}, \dots, w_{k,i, r}^{(l-1)} \in \mathbb{R}$
$(l=2, \dots, L)$
and
\begin{equation}
f_{k,i}^{(1)}(x) = 
f_{\bw,k,i}^{(1)}(x) = 
\sigma \left(\sum_{j=1}^d w_{k,i,j}^{(0)}\cdot x^{(j)} + w_{k,i,0}^{(0)} \right)
\end{equation}
for some $w_{k,i,0}^{(0)}, \dots, w_{k,i,d}^{(0)} \in \mathbb{R}$,
where the weight vector $\bw$
  satisfies
  \begin{equation}
    \label{le2ko24eq1}
    \sum_{j=1}^{K_n} |w_{1,1,j}^{(L)}| \leq C,
    \end{equation}
  \begin{equation}
    \label{le2ko24eq2}
    |w_{k,i,j}^{(l)}| \leq B \quad (k \in \{1, \dots, K_n\},
    i,j \in \{1, \dots, r\}, l \in \{1, \dots, L-1\})
    \end{equation}
and
  \begin{equation}
    \label{le2ko24eq3}
    |w_{k,i,j}^{(0)}| \leq A \quad (k \in \{1, \dots, K_n\},
    i \in \{1, \dots, r\}, j \in \{1, \dots,d\}).
  \end{equation}
  Then we have for any $1 \leq p < \infty$, $0 < \epsilon < 1$ and
  $x_1^n \in \M^n$
  \begin{eqnarray*}
&&\Nu_p \left(
\epsilon, \{ T_\kappa f \, : \, f \in \F \}, x_1^n
\right)
\leq \left(\const\cdot \frac{\kappa^p} {\epsilon^p}\right)^{\const \cdot B^{(L-1)\cdot d^*} \cdot A^{d^*} \cdot \left(\frac{C}{\epsilon}\right)^{d^*/k}+ \const
}.
\\
  \end{eqnarray*}
  
  \end{lemma}

\noindent
    {\bf Remark 3.} It follows from Lemma 12 in Kohler (2026) that the
    assertion also holds if $d^*=d$ and $\M$ is contained in some
    compact subset of $\R^d$ (but not necessarily a Lipschitz-manifold).\\

To prove Lemma \ref{le2ko24} we need the following result which
is a slight modification of Lemma 1 a) in Kohler, Langer and Reif (2023).

\begin{lemma}
  \label{le0.1}
  Let $\M$ be a $d^*$-dimensional Lipschitz-manifold.
  Let $h \in (0,1]$ and let $\P$ be a partition of $\R^d$ into cubes with side length $h$.
      Then 
      \[
      | \{ C \in \P \, : \, C \cap \M \neq \emptyset \} |
      \leq
      \const[c1le0] \cdot
      \left( \frac{1}{h} \right)^{d^*},
      \]
      where
      $\const[c1le0]{} = r \cdot (2 \cdot C_{\psi,2} \cdot \sqrt{d^*} +4)^{d^*}$.\\
  \end{lemma}
    {\bf Proof.}
    	For $k_1,\dots,k_{d^*}\in\{0,1,\dots,\lceil 1/h\rceil -1\}$ set
    	\[
    		A_{k_1,\dots,k_{d^*}}=
 			[k_1 \cdot h, \min\{(k_1+1) \cdot h, 1\})
           	\times \dots \times
            [k_{d^*} \cdot h, \min\{(k_{d^*}+1) \cdot h, 1\}).
    	\]
    	We can conclude from the definition of a Lipschitz-manifold
    	\begin{align*}
    		\M 
    		&=
    		\bigcup_{j=1}^r \M\cap U_j
    		=
    		\bigcup_{j=1}^r\psi_j\left((0,1)^{d^*}\right)
    		\\
    		&\subseteq
    		\bigcup_{j=1}^r
    		\bigcup_{
				k_1, \dots, k_{d^*} \in \{0,1, \dots, \lceil 1/h \rceil - 1 \}
    		}
    		\psi_j
    			\left(
		   			A_{k_1,\dots,k_{d^*}} 
    			\right).
    	\end{align*}
This implies 
\begin{align*}
&| \{ C \in \P \, : \, C \cap \M \neq \emptyset \} |
\\
&\leq \sum_{j=1}^{r} \sum_{k_1=0}^{\lceil \frac{1}{h}\rceil - 1} \dots \sum_{k_{d^*}=0}^{\lceil \frac{1}{h}\rceil -1}
|\{ C \in \P \, : \, C \cap
\psi_j
\left(
	A_{k_1,\dots,k_{d^*}}
\right)
\neq \emptyset \}|
\\
&\leq r\cdot 2^{d^*}\cdot \left(\frac{1}{h}\right)^{d^*}\cdot
\max_{j=1, \dots, r}
|\{ C \in \P \, : \, C \cap
\psi_j
\left(
	A_{k_1,\dots,k_{d^*}}
\right)
\neq \emptyset \}|
\end{align*}  
and it suffices to show 
    \begin{eqnarray}
      \label{ple0eq1}
      &&
\max_{j=1, \dots, r}
	  |\{ C \in \P \, : \, C \cap
	  \psi_j
	  \left(
	  	A_{k_1,\dots,k_{d^*}}
	  \right)
      \neq \emptyset \}| 
      \leq (C_{\psi,2} \cdot \sqrt{d^*} +2)^{d^*}.
    \end{eqnarray}
Fix $j \in \{1, \dots, r\}$.
    Each point of a cube
    \[
    [k_1 \cdot h, (k_1+1) \cdot h)
           \times \dots \times
                  [k_{d^*} \cdot h, (k_{d^*}+1) \cdot h)
    \]
    has a distance of at most $\frac{1}{2}\cdot \sqrt{d^*}\cdot h$ from its center.
    Hence, $A_{k_1,\dots,k_{d^*}}$ is contained in a ball of radius $\frac{1}{2}\cdot \sqrt{d^*}\cdot h$, and by Lipschitz continuity $\psi_j (A_{k_1,\dots,k_{d^*}})$ is contained in a cube of side length
        $C_{\psi,2} \cdot \sqrt{d^*} \cdot h$. The number of cubes from $\P$ that intersect that cube is bounded from above by
\[
\left(
	\frac{C_{\psi,2} \cdot \sqrt{d^*} \cdot h}{h} +2
\right)^{d^*}
= (C_{\psi,2} \cdot \sqrt{d^*}+2)^{d^*},
\]
since each cube from the partition $\P$ has side length $h$.
     \hfill $\Box$

\noindent
    {\bf Proof of Lemma \ref{le2ko24}.}
    Let $R>0$ be such that $\M \subseteq \{x \in \R^d \, : \, \|x\| \leq R\}$.
    The first step of the proof is to show, that
    for any $f_{\bw} \in \F$, any $x \in \R^d$ with $\|x\| \leq R$, any $k\in\N$ and any $s_1, \dots, s_k \in \{1, \dots, d\}$
    \begin{equation}
      \label{ple2ko24eq1}
      \left|
      \frac{\partial^k f_{\bw}}{\partial x^{(s_1)} \dots \partial x^{(s_k)}} (x)
      \right| \leq \const \cdot C \cdot B^{(L-1) \cdot k} \cdot A^k =: c
      .
      \end{equation}
This follows from the first step of the proof of Lemma 12 in Kohler (2024).    

In the second step of the proof we will show
\begin{equation}
	\label{ple12eq3}
	\Nu_p \left(
	\epsilon, \{T_\kappa f \, : \, f \in \F \}, x_1^n
	\right)
	\leq
	\Nu_p  \left(
	\frac{\epsilon}{2}, T_\kappa \G\circ \Pi, x_1^n
	\right),     
\end{equation}
where $\G$ denotes the set of all polynomials of degree less than or equal to $k-1$
and $\Pi$ is a partition of $\mathbb{R}^d$ into cubes of side length
$\left(\const[c1le12]\cdot \frac{\epsilon}{c}\right)^{1/k}$,
where $\const[c1le12]=\const[c1le12](d,k)>0$ is a suitable small constant.
Here $T_\kappa \G\circ \Pi$ is the class of all functions
whose restriction on each cube in $\Pi$ lies in $T_\kappa \G$.

Let  $(T f_\bw)_{k-1,u}$ be the multivariate Taylor polynomial of $f_\bw$ of degree $k-1$ around $u\in \R^d$, i.e.
\begin{align*}
      &(Tf)_{k-1,u}(x)
      \\
      &
            =
      \sum_{j_1, \dots, j_d \in \N_0, \atop
        j_1 + \dots + j_d \leq k-1}
            \frac{1}{j_1! \cdots j_d!}\cdot
      \frac{
\partial^{j_1+\dots + j_d} f
      }{
\partial^{j_1} x^{(1)} \dots \partial^{j_d} x^{(d)}
      }
      (u)
      \cdot
      (x^{(1)}-u^{(1)})^{j_1}  \cdots  (x^{(d)}-u^{(d)})^{j_d}.
\end{align*}
For each $I\in \Pi$ fix some $u\in I$. By a multivariate Taylor theorem  we get
as in the proof of Lemma 1 in Kohler (2014)
for each $x\in I$
\begin{eqnarray*}
  &&
  |f_\bw(x)-(T f_\bw)_{k-1,u}(x)|
  \\
  &&
  =
  \Bigg|
  f_\bw(x)-(T f_\bw)_{k-2,u}(x)
  \\
  &&
  \quad
  -
  \sum_{j_1, \dots, j_d \in \N_0, \atop
        j_1 + \dots + j_d = k-1}
            \frac{1}{j_1! \cdots j_d!}\cdot
      \frac{
\partial^{j_1+\dots + j_d} f_\bw
      }{
\partial^{j_1} x^{(1)} \dots \partial^{j_d} x^{(d)}
      }
      (u)
      \cdot
      (x^{(1)}-u^{(1)})^{j_1}  \cdots  (x^{(d)}-u^{(d)})^{j_d}
  \Bigg|
  \\
  &&
  =
  \Bigg|
\sum_{j_1, \dots, j_d \in \N_0, \atop
        j_1 + \dots + j_d = k-1}
\frac{k-1}{j_1! \cdots j_d!}\cdot
\int_0^1
(1-t)^{k-2}
\cdot
      \frac{
\partial^{j_1+\dots + j_d} f_\bw
      }{
\partial^{j_1} x^{(1)} \dots \partial^{j_d} x^{(d)}
      }
      (u+t \cdot (x-u)) \, dt
      \\
      &&
      \hspace*{4cm}
      \cdot
      (x^{(1)}-u^{(1)})^{j_1}  \cdots  (x^{(d)}-u^{(d)})^{j_d}
  \\
  &&
  \quad
  -
  \sum_{j_1, \dots, j_d \in \N_0, \atop
        j_1 + \dots + j_d = k-1}
            \frac{1}{j_1! \cdots j_d!}\cdot
      \frac{
\partial^{j_1+\dots + j_d} f_\bw
      }{
\partial^{j_1} x^{(1)} \dots \partial^{j_d} x^{(d)}
      }
      (u)
      \cdot
      (x^{(1)}-u^{(1)})^{j_1}  \cdots  (x^{(d)}-u^{(d)})^{j_d}
      \Bigg|
  \\
  &&
  =
  \Bigg|
\sum_{j_1, \dots, j_d \in \N_0, \atop
        j_1 + \dots + j_d = k-1}
\frac{k-1}{j_1! \cdots j_d!}
\cdot
      (x^{(1)}-u^{(1)})^{j_1}  \cdots  (x^{(d)}-u^{(d)})^{j_d}
\\
&&
\quad
\cdot
\int_0^1
(1-t)^{k-2}
\cdot
\left(
\frac{
\partial^{j_1+\dots + j_d} f_\bw
      }{
\partial^{j_1} x^{(1)} \dots \partial^{j_d} x^{(d)}
      }
(u+t \cdot (x-u))
-
      \frac{
\partial^{j_1+\dots + j_d} f_\bw
      }{
\partial^{j_1} x^{(1)} \dots \partial^{j_d} x^{(d)}
      }
      (u)
\right)
\, dt
\Bigg|
\\
&&
\leq
\sum_{j_1, \dots, j_d \in \N_0, \atop
        j_1 + \dots + j_d = k-1}
\frac{k-1}{j_1! \cdots j_d!}
\cdot
      |x^{(1)}-u^{(1)}|^{j_1}  \cdots  |x^{(d)}-u^{(d)}|^{j_d}
\\
&&
\quad
\cdot
\int_0^1
(1-t)^{k-2}
\cdot
\left|
\frac{
\partial^{j_1+\dots + j_d} f_\bw
      }{
\partial^{j_1} x^{(1)} \dots \partial^{j_d} x^{(d)}
      }
(u+t \cdot (x-u))
-
      \frac{
\partial^{j_1+\dots + j_d} f_\bw
      }{
\partial^{j_1} x^{(1)} \dots \partial^{j_d} x^{(d)}
      }
      (u)
\right|
\, dt.
  \end{eqnarray*}
By the mean value theorem and
(\ref{ple2ko24eq1}) we get for any $j_1, \dots, j_d \in \N_0$
with $j_1 + \dots + j_d = k-1$
\begin{eqnarray*}
  &&
\left|
\frac{
\partial^{j_1+\dots + j_d} f_\bw
      }{
\partial^{j_1} x^{(1)} \dots \partial^{j_d} x^{(d)}
      }
(u+t \cdot (x-u))
-
      \frac{
\partial^{j_1+\dots + j_d} f_\bw
      }{
\partial^{j_1} x^{(1)} \dots \partial^{j_d} x^{(d)}
      }
      (u)
      \right|
      \\
      &&
      \leq
      \sum_{i=1}^d
      \Bigg|
\frac{
\partial^{j_1+\dots + j_d} f_\bw
      }{
\partial^{j_1} x^{(1)} \dots \partial^{j_d} x^{(d)}
      }
(u^{(1)}, \dots, u^{(i-1)},  u^{(i)}+t \cdot (x^{(i)}-u^{(i)}), \dots,
u^{(d)}+t \cdot (x^{(d)}-u^{(d)}))
\\
&&
\quad
-
\frac{
\partial^{j_1+\dots + j_d} f_\bw
      }{
\partial^{j_1} x^{(1)} \dots \partial^{j_d} x^{(d)}
      }
(u^{(1)}, \dots, u^{(i)},  u^{(i+1)}+t \cdot (x^{(i+1)}-u^{(i+1)}), \dots,
u^{(d)}+t \cdot (x^{(d)}-u^{(d)}))
\Bigg|
\\
&&
\leq \sum_{i=1}^d c \cdot |u^{(i)}+t \cdot (x^{(i)}-u^{(i)}) - u^{(i)}|
\\
&&
= c \cdot t \cdot \sum_{i=1}^d |x^{(i)}-u^{(i)}|.
  \end{eqnarray*}
Hence,
\begin{align*}
  &|f_\bw(x)-(T f_\bw)_{k-1,u}(x)|
  \\
    & 
    \leq
    \sum_{j_1, \cdot \dots \cdot, j_d \in \N_0, \atop
    	j_1+\dots + j_d = k-1}
    	\frac{k-1}{j_1!\cdots j_d!}
    	\cdot
    	c
        \cdot
        \sum_{i=1}^d \left|x^{(i)}-u^{(i)}\right|
        \cdot
        \left|x^{(1)}-u^{(1)}\right|^{j_1} \cdots  \left|x^{(d)}-u^{(d)}\right|^{j_d}
        \\
        &
        \hspace*{3cm}
        \cdot
        \int_0^1 t \cdot (1-t)^{k-2} dt 
      \\
      &
      =
      \sum_{j_1, \cdot \dots \cdot, j_d \in \N_0, \atop
    	j_1+\dots + j_d = k-1}
    	\frac{1}{j_1!\cdots j_d! \cdot k}
    	\cdot
    	c
        \cdot
        \sum_{i=1}^d \left|x^{(i)}-u^{(i)}\right|
        \cdot
        \left|x^{(1)}-u^{(1)}\right|^{j_1} \cdots  \left|x^{(d)}-u^{(d)}\right|^{j_d}
        \\
        &
        \leq
        \sum_{j_1, \cdot \dots \cdot, j_d \in \N_0, \atop
    	j_1+\dots + j_d = k}
    	\frac{1}{j_1!\cdots j_d! }
    	\cdot
    	c
        \cdot
        \left|x^{(1)}-u^{(1)}\right|^{j_1} \cdots  \left|x^{(d)}-u^{(d)}\right|^{j_d}
        \\
        &
        =
        \frac{c}{k!} \cdot \left(
        \sum_{i=1}^d \left|x^{(i)}-u^{(i)}\right|
        \right)^k
        \\
        &
        \leq
        c \cdot \frac{d^k}{k!} \cdot \|x-u\|_\infty^k
        \leq
        c \cdot \frac{d^k}{k!} \cdot \const[c1le12]\cdot \frac{\epsilon}{c}
        =
        \const[c1le12] \cdot \frac{d^k}{k!} \cdot \epsilon,
\end{align*}
where we have used the multinomial theorem.
By repeating this argument for every cube $I\in\Pi$ and by choosing $\const[c1le12]$ small enough we can see that for each $f_\bw$ we can find $g \in \mathcal{G}\circ \Pi$ such that
\[
|f_\bw(x)-g(x)|\leq \frac{\epsilon}{2}
\]
  holds for all $x\in \R^d$, which implies (\ref{ple12eq3}).  
   The last step is to complete the proof of Lemma \ref{le2ko24}.
   For 
   \[
   \Pi^*:=\{I\in\Pi:I\cap \M\neq \emptyset\}
   \cup \left\{\mathbb{R}^d\setminus 
   \bigcup_{I\in\Pi:I\cap \M\neq \emptyset}I\right\}
   \]
   we have
   \begin{align}\label{ple12eq3b}
   	\Nu_p  \left(
	\frac{\epsilon}{2}, T_\kappa \G\circ \Pi, x_1^n
	\right)
	\leq
	\Nu_p  \left(
	\frac{\epsilon}{2}, T_\kappa \G\circ \Pi^*, x_1^n
	\right),
   \end{align}
   since $x_1^n\in\M^n$.
   There are
   \[
   \binom{d+k-1}{d}
   \]
   many monomials in $d$ variables of degree at most $k-1$.
   This together with Lemma \ref{le0.1} allows us to conclude that
   $\mathcal{G}\circ \Pi^*$ is a linear vector space of dimension
    \[
	    \binom{d+k-1}{d} \cdot \left|\Pi^*\right| \leq \const[c4le12] \cdot
            \left(\frac{c}{\epsilon}\right)^{d^*/k}.
    \]
   With this we derive from
   Theorem 9.4 and Theorem 9.5 in Györfi et al. (2002)
   \begin{equation}
     \label{ple12eq3c}
    	\mathcal{N}_p(\frac{\epsilon}{2}, T_\kappa \mathcal{G} \circ \Pi^*, x_1^n)
    	\leq 3 \left(\frac{2e(2\kappa)^p}{(\epsilon/2)^p}\log\left(\frac{3e(2 \kappa)^p}{(\epsilon/2)^p}\right)\right)^{\const[c4le12] \cdot \left(\frac{c}{\epsilon}\right)^{{d^*}/k}+1}.
    \end{equation}
    The claim follows from (\ref{ple12eq3}), (\ref{ple12eq3b}) and (\ref{ple12eq3c}).

\hfill $\Box$
  
        \subsection{Proof of Theorem \ref{th1}}
        \label{se5se4}
        We prove a) and b) at the same time. To do this, we let $d^* \in \{1, \dots, d\}$, assume
        \[
supp(X) \subseteq \M,
        \]
        and assume that for $d^*<d$ $\M$ is a $d^*$-dimensional Lipschitz-manifold, and for $d^*=d$ we assume that $\M$ is contained in a compact subset of $\R^d$.

        We give a proof similar to the proof of Theorem 1 in Kohler (2026).
Throughout the proof we assume w.l.o.g. that $n$ is sufficiently large
and that $\|m\|_\infty\leq \kappa_n$ holds. Since $\M$ is bounded, we
can choose $A\geq 1$ with $\M\subseteq [-A,A)^d$.
Set
\[
K=\left\lceil \const\cdot n^{\frac{1}{2p+d^*}}\right\rceil, 
\]
\[
\tilde{K}_n = r\cdot K^d\text{ and}
\]
\[
N_n = \left\lceil \const \cdot n^{1+\frac{4p+3d}{2p+d^*}}\right\rceil. 
\]
Set
\[
J_n=K_n\cdot \left(
r+2
+(L-2)\cdot r \cdot (r+1)
+ r\cdot (d+1)
\right)
\]
and
\[
J^*_n=N_n\cdot \tilde{K}_n\cdot \left(
r+2
+(L-2)\cdot r \cdot (r+1)
+ r\cdot (d+1)
\right).
\]
By using Theorem \ref{th3} with $K$, $A$, $L$ and $r$
we know that there exists a weight vector $$\bw^*\in \R^{J^*_n}$$ of a neural network
\[
f_{\bw^*}(x)=\sum_{k=1}^{N_n\cdot\tilde{K}_n}\left(\bw^*\right)_{1,1,k}^{(L)}
\cdot f^{(L)}_{\bw^*,k,1}(x),
\]
where each $f^{(L)}_{\bw^*,k,1}$ is a subnetwork 
consisting of $L$ layers and $r$ neurons per layer, such that
\begin{equation}\label{pth1eq1}
\sup_{x\in [-A,A)^d} |f_{\bw^*}(x)-m(x)|\leq \frac{\const}{\tilde{K}_n^{p/d}}
\end{equation}
and
\begin{equation*}
|\left(\bw^*\right)_{1,1,k}^{(L)}|\leq \frac{\const[c6pth1]\cdot \tilde{K}_n^{(q+d)/d}}{N_n}\quad(k=1,\dots,N_n\cdot\tilde{K}_n).
\end{equation*}
Note that here we replace each of the $f_k$'s in the outer sum of the space $\HH$ of Theorem \ref{th3} with
\[
f_k=\frac{1}{N_n}\sum_{i=1}^{N_n}f_k.
\]
Furthermore, the weights of this network satisfy
\[
|\left(\bw^*\right)_{k,i,j}^{(l)}| \leq \const \quad \mbox{for } l=1, \dots, L-1
\]
and
\[
|\left(\bw^*\right)_{k,i,j}^{(0)}| \leq \const \cdot (\log n) \cdot n^\tau.
\]
Set
\[
\epsilon_n = \frac{\const}{n\cdot\sqrt{N_n\cdot\tilde{K}_n}}\geq \frac{\const[c1pth1]}{n^{(2p+2d)\cdot \tau+1.5}},
\]
where the inequality follows from
\[
n\cdot \sqrt{N_n\cdot \tilde{K}_n}
\leq \const[cn1pth1]\cdot n^{\frac{3}{2}}\cdot n^{\frac{4p+3d}{4p+2d^*}}\cdot n^{\frac{d}{4p+2d^*}}
=\const[cn1pth1]\cdot n^{(2p+2d)\cdot \tau+1.5}
.
\]
Let $A_n$ be the event such that firstly there exist pairwise distinct $j_1,\dots,j_{N_n\cdot\tilde{K}_n}\in\{1,\dots,K_n\}$ such that the weight vector $\bw^{(0)}\in\R^{J_n}$ satisfies
\[
| (\bw^{(0)})_{j_s,k,i}^{(l)}-(\bw^*)_{s,k,i}^{(l)}| \leq \epsilon_n
\quad \mbox{for all } k,i\mbox{ and }l \in \{0, \dots, L-1\},
s \in \{1, \dots, N_n \cdot \tilde{K}_n \}
\]
and such that secondly
\[
\max_{i=1, \dots, n} |Y_i|^2 \leq \kappa_n
\]
holds.
In the proof we split the $L_2$ error of $m_n$ into a sum of several terms and
bound each term separately. In the following we set 
\[
m_{\kappa_n}(x)=\EXP\{ T_{\kappa_n} Y | X=x \}.
\]
Then we have
\begin{eqnarray*}
&&
\int | m_n(x)-m(x)|^2 \PROB_X (dx)
\\
&&
=
\left(
\EXP \left\{ |m_n(X)-Y|^2 | \D_n \right\}
-
\EXP \{ |m(X)-Y|^2\}
\right)
\cdot 1_{A_n}
\\
&&
\quad
+
\int | m_n(x)-m(x)|^2 \PROB_X (dx)
\cdot 1_{A_n^c}
\\
&&
=
\Big[
\EXP \left\{ |m_n(X)-Y|^2 | \D_n \right\}
-
\EXP \{ |m(X)-Y|^2\}
\\
&&
\hspace*{2cm}
- \left(
\EXP \left\{ |m_n(X)-T_{\kappa_n} Y|^2 | \D_n \right\}
-
\EXP \{ |m_{\kappa_n}(X)- T_{\kappa_n} Y|^2\}
\right)
\Big] \cdot 1_{A_n}
\\
&&
\quad +
\Big[
\EXP \left\{ |m_n(X)-T_{\kappa_n} Y|^2| \D_n \right\}
-
\EXP \{ |m_{\kappa_n}(X)- T_{\kappa_n} Y|^2\}
\\
&&
\hspace*{2cm}
-
2 \cdot \frac{1}{n} \sum_{i=1}^n
\left(
|m_n(X_i)-T_{\kappa_n} Y_i|^2
-
|m_{\kappa_n}(X_i)- T_{\kappa_n} Y_i|^2
\right)
\Big] \cdot 1_{A_n}
\\
&&
\quad
+\Big[
2 \cdot \frac{1}{n} \sum_{i=1}^n
|m_n(X_i)-T_{\kappa_n} Y_i|^2
-
2 \cdot \frac{1}{n} \sum_{i=1}^n
|m_{\kappa_n}(X_i)- T_{\kappa_n} Y_i|^2
\\
&&
\hspace*{2cm}
- \left(
2 \cdot \frac{1}{n} \sum_{i=1}^n
|m_n(X_i)-Y_i|^2
-
2 \cdot \frac{1}{n} \sum_{i=1}^n
|m(X_i)- Y_i|^2
\right)
\Big] \cdot 1_{A_n}
\\
&&
\quad
+
\Big[
2 \cdot \frac{1}{n} \sum_{i=1}^n
|m_n(X_i)-Y_i|^2
-
2 \cdot \frac{1}{n} \sum_{i=1}^n
|m(X_i)- Y_i|^2
\Big] \cdot 1_{A_n}
\\
&&
\quad
+
\int | m_n(x)-m(x)|^2 \PROB_X (dx)
\cdot 1_{A_n^c}
\\
&&
=: \sum_{j=1}^5 T_{j,n}.
\end{eqnarray*}
In the first step of the proof we show
\[
\EXP \{T_{j,n}\} \leq \const \cdot \frac{\log n}{n} \quad
\mbox{for } j \in \{1,3\}.
\]
This follows as in the proof of Lemma 1 in Bauer and Kohler (2019).

\noindent
In the second step of the proof we show
\[
\EXP \{T_{5,n}\} \leq \const \cdot \frac{(\log n)^2}{n^2}.
\]
Since we assume $\|m\|_\infty\leq \kappa_n$, we have
\begin{align*}
    \int |m_n(x)-m(x)|^2 \PROB_X (dx)
    \leq \int |2\cdot \kappa_n|^2 \PROB_X (dx)
    \leq
4 \cdot \const[c2th1]^2 \cdot (\log
n)^2
\end{align*}
and it suffices to show
\begin{equation}
\label{pth1eq2}
\PROB(A_n^c) \leq \frac{\const}{n^2}.
\end{equation}
We consider the initial
choice of the weights for the $K_n$ fully connected neural networks sequentially. The weight in the first of these networks differs in all components
in layers $l=1, \dots, L-1$
by at most $\epsilon_n$ from $\left(\bw^*\right)^{(l)}_{1,i,j}$ with probability bounded from below by
\begin{eqnarray*}
&&
  \left( \frac{\epsilon_n}{2\cdot \const[c1]}
\right)^{r \cdot (r+1) \cdot (L-1)}
\cdot
\left(
\frac{\epsilon_n}{2\cdot \const[c2] \cdot (\log n) \cdot n^\tau}
\right)^{r \cdot (d+1)}
\\
&&
\geq
  \left(\frac{\const[c1pth1]}{2 \cdot \const[c1] \cdot n^{(2p+2d)\cdot \tau+1.5}}
\right)^{r \cdot (r+1) \cdot (L-1)}
\cdot
\left(
\frac{\const[c1pth1]}{2 \cdot \const[c2] \cdot (\log n) \cdot n^\tau \cdot n^{(2p+2d)\cdot \tau+1.5}}
\right)^{r \cdot (d+1)}
\\
&&
\geq
n^{-((2p+2d)\cdot \tau+1.5) \cdot r \cdot (r+1) \cdot (L-1) - ((2p+2d)\cdot \tau+1.5)\cdot r \cdot (d+1) -  \tau\cdot r \cdot (d+1) - 0.5}
\\
&&
=n^{-\eta -0.5},
\end{eqnarray*}
where
\[
\eta
=
((2p+2d)\cdot \tau+1.5) \cdot r \cdot (r+1) \cdot (L-1) + ((2p+2d)\cdot \tau+1.5)\cdot r \cdot (d+1) + \tau\cdot r \cdot (d+1). 
\]
Therefore, the probability that none of the first $n^{\eta + 1}$
neural networks satisfy this condition is bounded  from above by
\begin{align*}
&\left(1-n^{-\eta -0.5}\right)^{n^{\eta +1}}
\leq
\left(\exp\left(-n^{-\eta - 0.5}\right)\right)^{n^{\eta +1}}
=
\exp\left(-n^{0.5}\right).
\end{align*}
Assumption (\ref{th1eq1}) implies that $K_n\geq n^{\eta +1}\cdot N_n\cdot \tilde{K}_n$ for all sufficiently large $n$. Consequently, we can conclude, that the probability that there exists $s\in\{1,\dots,N_n\cdot\tilde{K}_n\}$ such
that all $K_n$ weight vectors differ from
$((\bw^{*})_{s,k,i}^{(l)})_{k,i,l}$
in at least one component by more than $\epsilon_n$ is  bounded from above by

\begin{eqnarray*}
N_n \cdot \tilde{K}_n \cdot \exp( -  n^{0.5})
&&
\leq  \const \cdot n^{(4p+3d)\cdot \tau+1}\cdot n^{d\cdot\tau} \cdot \exp( - n^{0.5})
\leq \frac{\const[c3pth1]}{n^2}.
\end{eqnarray*}
We conclude with Markov's inequality, (A1) and $\const[c2th1]\cdot \const[c1th1]\geq 3$
\begin{eqnarray*}
\PROB(A_n^c)
&\leq&
\frac{\const[c3pth1]}{n^2}
+
\PROB\{ \max_{i=1, \dots, n} Y_i^2 > \kappa_n
\}
 \leq
\frac{\const[c3pth1]}{n^2}
+
n \cdot\PROB\{  Y^2 > \kappa_n
\}
\\
&=&
\frac{\const[c3pth1]}{n^2}
+
n\cdot \PROB\{\exp(\const[c1th1]\cdot Y^2)>\exp(\const[c1th1]\cdot\kappa_n)\}
 \leq 
\frac{\const[c3pth1]}{n^2}
+
n \cdot
\frac{\EXP\{ \exp(\const[c1th1] \cdot Y^2)\}}{\exp( \const[c1th1] \cdot \kappa_n)}\\
 &=& 
\frac{\const[c3pth1]}{n^2}
+
n \cdot
\frac{\EXP\{ \exp(\const[c1th1] \cdot Y^2)\}}{ \exp( \const[c2th1]\cdot \const[c1th1] \cdot\log(n))}
= \frac{\const[c3pth1]}{n^2}
+
n\cdot\frac{\EXP\{ \exp(\const[c1th1] \cdot Y^2)\}}{n^{\const[c2th1]\cdot \const[c1th1]}}
\leq
\frac{\const}{n^2},
\end{eqnarray*}
for $n$ sufficiently large.

Let $\epsilon >0$ be arbitrary.
The third step is to show
\[
\EXP \{T_{2,n}\} \leq
\const[cT2pth1] \cdot
\frac{ n^{\tau \cdot d^* + \epsilon}}{n}
=\const[cT2pth1]\cdot n^{-\frac{2p}{2p+d^*}+\epsilon}
.
\]
Let $\W_n$ be the set of all weight vectors
$\bw=(w_{i,j,k}^{(l)})_{i,j,k,l}\in \R^{J_n}$ which satisfy
\[
| w_{1,1,k}^{(L)}| \leq \const[c64] \quad (k=1, \dots, K_n),
\]
\[
|w_{i,j,k}^{(l)}| \leq \const[c65] \quad (l=1, \dots, L-1)
\]
and
\[
|w_{i,j,k}^{(0)}| \leq \const[c66] \cdot (\log n) \cdot n^\tau.
\]
By Lemma \ref{le3}, Lemma \ref{le3Ko24} and Lemma \ref{le5Ko24} we can conclude 
\begin{equation*}
\| \bw^{(t)}-\bw^{(0)}\| \leq \const \quad (t=1,
\dots, t_n)
\end{equation*}
on $A_n$. This is done similarly as in the proof of Theorem 2 and uses
\[
t_n\cdot\lambda_n\cdot\kappa_n \leq \const.
\]
The choice of $\bw^{(0)}$ implies for
$\const[c64],\const[c65],\const[c66]$ sufficiently large that
we have on $A_n$
\[
\bw^{(t)} \in \W_n \quad (t=0, \dots, t_n).
\]
This means that 
\[
m_n\in\F_n = \left\{ T_{\kappa_n} f_\bw \quad : \quad \bw \in \W_n \right\}
\]
and we conclude with Lemma \ref{lea3} for $u_n>0$

\begin{align*}
&\EXP\{T_{2,n}\}\leq
4\cdot\kappa_n^2\cdot \frac{n}{(\log n)^2}\cdot \beta_{\lceil (\log
  n)^2\rceil}
((X_1,Y_1), (X_2,Y_2), \dots)+u_n
\\
&\quad\quad
+\int_{u_n}^\infty 
14\cdot \sup_{x_1^n\in \supp(X)^n}\mathcal N_1\left(\frac{u}{80\cdot \kappa_n},\F_n,x_1^n\right)\cdot\exp\left(-\const\cdot\frac{u\cdot n}{\kappa_n^2\cdot (\log n)^2}\right) \, du.
\end{align*}

Lemma \ref{le2ko24} (in case $d^*<d$) and Remark 3 (in case $d^*=d$) imply
for $x_1^n\in\M^n$
\begin{eqnarray*}
&&
\Nu_1 \left(
\frac{u}{80\cdot \kappa_n} , \F_n
, x_1^n
\right)
\leq
\left(
\frac{\const}{u/(80\cdot \kappa_n)}
\right)^{\const
\cdot  (\const)^{(L-1) \cdot d^*}
\cdot (\log n)^{d^*}  n^{\tau \cdot d^*} 
\cdot
   \left(\frac{K_n \cdot \const}{u/(80\cdot \kappa_n)}\right)^{d^*/k} + \const
  }
\end{eqnarray*}
and by choosing $k$ large enough we get for $u >\frac{\const}{n}$
\[
\Nu_1 \left(
\frac{u}{80\cdot \kappa_n} , \F_n
, x_1^n
\right)
\leq
\const[c4pth1]\cdot n^{\const[c5pth1]\cdot n^{\tau\cdot d^*+\epsilon/2}}.
\]
We have by (A6)
\[
\beta_{\lceil (\log n)^2\rceil}
((X_1,Y_1), (X_2,Y_2), \dots)
\leq \const[71]\cdot e^{-\const[72]\cdot \lceil (\log n)^2\rceil}\leq \const[71]\cdot e^{-\const[72]\cdot (\log n)^2} 
\]

For $u_n\geq \const/n$ we conclude 
\begin{align*}
\EXP\{T_{2,n}\}&\leq
4\cdot\const[71]\cdot\kappa_n^2\cdot \frac{n}{(\log n)^2}\cdot e^{-\const[72]\cdot (\log n)^2} +u_n
\\
&\quad\quad
+\int_{u_n}^\infty 
14\cdot\const[c4pth1]\cdot n^{\const[c5pth1]\cdot n^{\tau\cdot d^*+\epsilon/2}}
\cdot\exp\left(-\const[73]\cdot\frac{u\cdot n}{\kappa_n^2\cdot (\log n)^2}\right) \, du
\\
&\leq
4\cdot\const[71]\cdot\kappa_n^2\cdot \frac{n}{(\log n)^2}\cdot e^{-\const[72]\cdot (\log n)^2} +u_n
\\
&\quad\quad +
\const[74]\cdot n^{\const[c5pth1]\cdot n^{\tau\cdot d^*+\epsilon/2}}
\cdot\exp\left(-\const[73]\cdot\frac{u_n\cdot n}{\kappa_n^2\cdot (\log n)^2}\right)
\cdot \frac{\kappa_n^2\cdot (\log n)^2}{n}
\end{align*}
The result of the third step of the proof follows by setting
\[
u_n = \frac{\kappa_n^2 \cdot (\log n)^2}{\const[73]\cdot n}
\cdot
\const[c5pth1]
\cdot
 n^{\tau \cdot d^* +
    \epsilon/2}
\cdot \log n.
\]
In the last step of the proof we will  show 
\begin{eqnarray*}
&&
  \EXP \{ T_{4,n} \} 
  \leq \const \cdot n^{- \frac{2p}{2p+d^*}}.
\end{eqnarray*}
Using
\[
|T_{\kappa_n} z - y| \leq |z-y|
\quad \mbox{for } |y| \leq \kappa_n
\]
we get
\begin{eqnarray*}
 \frac{T_{4,n}}{2}
&&=
\Big[ \frac{1}{n} \sum_{i=1}^n
|m_n(X_i)-Y_i|^2
-
 \frac{1}{n} \sum_{i=1}^n
|m(X_i)- Y_i|^2
\Big] \cdot 1_{A_n}
\\
&&
=
\Big[ \frac{1}{n} \sum_{i=1}^n
|T_{\kappa_n}f_{\bw^{(t_n)}}(X_i)-Y_i|^2
-
 \frac{1}{n} \sum_{i=1}^n
|m(X_i)- Y_i|^2
\Big] \cdot 1_{A_n}
\\
&&
\leq
\Big[
\frac{1}{n} \sum_{i=1}^n
|f_{\bw^{(t_n)}}(X_i)-Y_i|^2
-
 \frac{1}{n} \sum_{i=1}^n
|m(X_i)- Y_i|^2
\Big] \cdot 1_{A_n}
\\
&&
=
\Big[ F_n(\bw^{(t_n)})
-
 \frac{1}{n} \sum_{i=1}^n
|m(X_i)- Y_i|^2
\Big] \cdot 1_{A_n}.
\end{eqnarray*}
On $A_n$ let $\tilde{\bw}\in \R^{J_n}$ be an extension of $\bw^*\in \R^{J^*_n}$ onto $\R^{J_n}$, where we fill each new component with the value of $\bw^{(0)}$, i.e.
\begin{align*}
    (\tilde{\bw})_{j_s,k,i}^{(l)}&= (\bw^*)_{s,k,i}^{(l)}\mbox{ for }s\in\{1,\dots,N_n\cdot\tilde{K}_n\}\mbox{ and}\\
    (\tilde{\bw})_{j,k,i}^{(l)}&=({\bw^{(0)}})_{j,k,i}^{(l)}\mbox{ for }j\notin\{j_1,\dots,j_{N_n\cdot\tilde{K}_n}\}.
\end{align*}
The networks defined by $\bw^*$ and $\tilde{\bw}$ are identical, since the weights in the outer Layer of $\bw^{(0)}$ are all $0$.
On $A_n$ we have
\begin{eqnarray*}
\| \tilde{\bw} - \bw^{(0)} \|^2
&\leq&
\sum_{k=1}^{N_n\cdot\tilde{K}_n} |(\tilde{\bw})_{1,1,j_k}^{(L)}|^2 + N_n \cdot \tilde{K}_n
   \cdot L \cdot (r \cdot
   (r+d)) \cdot \epsilon_n^2
\\
&\leq&
N_n \cdot \tilde{K}_n \cdot \left(
\frac{\const[c6pth1] \cdot \tilde{K}_n^{(q+d)/d}}{N_n}
\right)^2
+
\const \cdot N_n \cdot \tilde{K}_n
   \cdot \epsilon_n^2
\\
&\leq&
\const\cdot \frac{\tilde{K}_n^{2\cdot (q+d)/d+1}}{N_n}
+\const[c7pth1]\cdot\frac{1}{n^2} 
\\
&\leq&
\const\cdot n^{-1-\frac{2p}{2p+d^*}}+\const[c7pth1]\cdot \frac{1}{n^2}
   \leq \const\cdot n^{-1-\frac{2p}{2p+d^*}}. 
\end{eqnarray*}
Here we used 
\begin{equation*}
\left(
	2\cdot \frac{q+d}{d}+1
\right)
\cdot \frac{d}{2p+d^*}
-
1
-
\frac{4p+3d}{2p+d^*}
=
-
1
-
\frac{4p-2q}{2p+d^*}
\leq
-1-\frac{2p}{2p+d^*}
.
\end{equation*}
We get with Theorem \ref{th2}
\begin{eqnarray*}
  &&
  \frac{T_{4,n}}{2}
  \\
  &&
 \leq
 \Bigg(
\frac{1}{n} \sum_{i=1}^n
|f_{\tilde{\bw}}(X_i)-Y_i|^2
+ \const \cdot (\log n)^2 \cdot \|\tilde{\bw}-\bw^{(0)}\|^2
+ \const \cdot \frac{(\log n)^{3}}{K_n^{3/2}}
\\
&&
\quad \quad
-
 \frac{1}{n} \sum_{i=1}^n
|m(X_i)- Y_i|^2
 \Bigg)
 \cdot 1_{A_n}
\\
&&
 \leq
\frac{1}{n} \sum_{i=1}^n
|f_{\bw^*}(X_i)-Y_i|^2
- \frac{1}{n} \sum_{i=1}^n
|m(X_i)- Y_i|^2
+ \const \cdot  (\log n)^2\cdot n^{-1-\frac{2p}{2p+d^*}}
\\
&&
\quad \quad
 +
 \frac{1}{n} \sum_{i=1}^n
|m(X_i)- Y_i|^2
 \cdot 1_{A_n^c}. 
\end{eqnarray*}
Using (\ref{pth1eq2}) and the C-S inequality, we can conclude
\begin{eqnarray*}
              &&
              \EXP\left\{\frac{1}{n} \sum_{i=1}^n
|f_{\bw^*}(X_i)-Y_i|^2
- \frac{1}{n} \sum_{i=1}^n
|m(X_i)- Y_i|^2
 +
 \frac{1}{n} \sum_{i=1}^n
|m(X_i)- Y_i|^2
 \cdot 1_{A_n^c}
\right\}
              \\
              &&
              \leq 
\EXP\{|f_{\bw^*}(X)-Y|^2\}- \EXP\left\{
|m(X)- Y|^2\right\}
+\frac{1}{n}\sum_{i=1}^n\sqrt{\EXP\{|m(X_i)-Y_i|^4\}}\cdot\sqrt{\PROB\{A_n^c\}}
 \\
                &&
                \leq
                 \int |f_{\bw*}(x)-m(x)|^2 \PROB_X(dx)
+\const\cdot \frac{(\log n)^2}{n}.
              \end{eqnarray*}
Here we have used that $\|m\|_\infty\leq \kappa_n=c_3\cdot \log n$ and that all moments of $Y$ are bounded
               due to $\EXP\{\exp(c_4\cdot Y^2)\}<\infty$.
Inequality (\ref{pth1eq1}) implies
\begin{align*}
    \EXP\{T_{4,n}\}
    &\leq
    \const\cdot\frac{1}{\tilde{K}_n^{2\cdot\frac{p}{d}}}+ \const \cdot  \frac{(\log n)^2}{n}
    \leq \const\cdot n^{-\frac{2p}{2p+d^*}}.
\end{align*}
Combing all steps, we get for all $\epsilon >0$
\begin{align*}
    \EXP\int |m_n(x)-m(x)|^2 \PROB_X(dx)\leq \const \cdot n^{-\frac{2p}{2p+d^*}+\epsilon}.
\end{align*}
\hfill $\Box$

\begin{appendix}
\section{Proof of Lemma \ref{lea2}}\label{AppendixSe1}
For the proof of Lemma \ref{lea2} we need the following lemma.

\begin{lemma}
\label{lea1}
Let
$N,d,K \in \N$, let
$X: \Omega \rightarrow \R^N$,
$Y: \Omega \rightarrow \R^d$,
$Z: \Omega \rightarrow \R^K$
and
$U: \Omega \rightarrow \R$
be Borel measurable random variables defined on a probability
space
$(\Omega,\A,\PROB)$ such that $(X,Y)$ and $U$ are independent and
$U$ is uniformly distributed on $[0,1]$.

Then there exists a probability space
$(\bar{\Omega},\bar{\A}, \bar{\PROB})$,
Borel measurable random variables
$\bar{X}: \bar{\Omega} \rightarrow \R^N$,
$\bar{ Y}, \bar{Y^*}: \bar{\Omega} \rightarrow \R^d$,
$\bar{Z}: \bar{\Omega} \rightarrow \R^K$,
$\bar{ U}: \bar{\Omega} \rightarrow \R$
and a Borel measurable function
$f:\R^N \times \R^d \times \R \rightarrow \R^d$
such that
\begin{equation}
\label{lea1eq1}
\PROB_{(X,Y,Z,U)}
=
\bar{\PROB}_{
(\bar{X},\bar{Y},\bar{Z},\bar{U})
},
\end{equation}
\begin{equation}
\label{lea1eq2}
\PROB_Y = \bar{\PROB}_{\bar{Y}^*},
\end{equation}
\begin{equation}
\label{lea1eq3}
\bar{Y}^* = f( \bar{X},\bar{Y},  \bar{U})  \quad a.s.
\end{equation}
\begin{equation}
\label{lea1eq4}
\bar{X}, \bar{Y}^* \mbox{ independent}
\end{equation}
and
\begin{eqnarray}
\label{lea1eq5}
\bar{\PROB}
\{ \bar{Y} \neq \bar{Y}^* \}
\leq
\EXP
\left\{
\esssup_{A \in \B_d}
\left|
\PROB\{  Y \in A | X \}
-
\PROB\{ Y \in A \}
\right|
\right\}.
\end{eqnarray}
\end{lemma}

\noindent
{\bf Remark 4.} In the proof we will show
\[
\bar{\PROB}
\{ \bar{Y} \neq \bar{Y}^* \}
\leq
\EXP
\left\{
\sup_{A \in \C}
\left|
\PROB\{  Y \in A | X \}
-
\PROB\{ Y \in A \}
\right|
\right\}
\]
for some countable subset $\C$ of $\B_d$.

\noindent
{\bf Proof.}
The proof is based on the proof sketch of Theorem 1 in Doukhan (1994),
Section 1.1.

Partition
$[-2^n,2^n]^d$
into $N_n-1=4^{n \cdot d}$ many cubes
$C_{1,n}$, \dots, $C_{N_n-1,n}$
 of side length
\[
\frac{2 \cdot 2^n}{4^n}
=
2
\cdot \left( \frac{1}{2} \right)^n,
\]
set
\[
C_{N_n,n} = \R^d \setminus [-2^n,2^n]^d,
\]
let
$y_{C_{1,n}}$, \dots, $y_{C_{N_n-1,n}}$ be the centers of
$C_{1,n}$, \dots, $C_{N_n-1,n}$, and set
\[
y_{C_{N_n,n}} = (2^n+1, \dots, 2^n+1)^T.
\]
Set
\[
Y_n(\omega) = y_{C_{i,n}} \quad \mbox{if } Y(\omega) \in C_{i,n}
\]
$(i=1, \dots, N_n)$.

Then
\begin{equation}
\label{plea1eq1}
Y_n \rightarrow^\PROB Y,
\end{equation}
since for any $\epsilon>0$ we have
\[
\limsup_{n \rightarrow \infty}
\PROB\{
\|Y_n-Y\|_\infty > \epsilon
\}
\leq
\limsup_{n \rightarrow \infty}
\PROB\{ Y \in C_{N_n,n} \}=0.
\]
For $x \in \R^d$ we define
\[
\lambda_i=\lambda_{i,n}(x)
=
\PROB\{ Y \in C_{i,n} | X=x\}
\quad \mbox{and} \quad
\mu_i=\mu_{i,n}=\PROB\{Y \in C_{i,n} \}
\]
for $i=1, \dots, N_n$, and reorder
$C_{1,n}$, \dots, $C_{N_n,n}$ (depending on $x$)
such that 
\[
\lambda_1 \leq \mu_1, \dots, \lambda_k \leq \mu_k,
\lambda_{k+1} > \mu_{k+1}, \dots, \lambda_{N_n} > \mu_{N_n}
\]
holds for some $k=k(x) \in \{1, \dots, N_n\}$. This means
\[
\PROB\{ Y \in C_{i,n} | X=x\} \leq \PROB\{ Y \in C_{i,n} \}
\quad \mbox{for } i=1, \dots, k
\]
and
\[
\PROB\{ Y \in C_{i,n} | X=x\} > \PROB\{ Y \in C_{i,n} \}
\quad \mbox{for } i=k+1, \dots, N_n.
\]
Given $X=x$ and $Y_n=y_{C_{i,n}}$ (the latter is equivalent to $Y \in
C_{i,n}$),
we choose the value of $Y_n^*$ as follows: In case
$\PROB\{Y \in C_{i,n}\} \geq \PROB\{Y \in C_{i,n}|X=x\}$
we set $Y_n^*=y_{C_{i,n}}$. And in case
$\PROB\{Y \in C_{i,n}\} < \PROB\{Y \in C_{i,n}|X=x\}$ we choose the
value of $Y_n^*$ randomly as follows: With probability
\[
\frac{(\mu_j-\lambda_j)\cdot (\lambda_i-\mu_i)}{a \cdot \lambda_i}
\]
we set it equal to $y_{C_{j,n}}$ for $j=1, \dots, k$, and with
probability
\[
\frac{\mu_i}{\lambda_i}
\]
we set it equal to $y_{C_{i,n}}$, where
\[
a=a_{n,x}= \sum_{l=1}^k (\mu_l-\lambda_l) = \sum_{l=k+1}^{N_n} (\lambda_l-\mu_l).
\]
(Here the last equality holds since $\sum_{l=1}^{N_n} \mu_l
=1=\sum_{l=1}^{N_n} \lambda_l$.)
This is possible since
\[
\sum_{j=1}^k \frac{(\mu_j-\lambda_j)\cdot (\lambda_i-\mu_i)}{a \cdot \lambda_i}
+\frac{\mu_i}{\lambda_i}=1.
\]
It is easy to see that there exists a Borel measurable function
$f_n: \R^N \times \R^d \times \R \rightarrow \R$ such that
\begin{equation}
\label{plea1eq2}
Y_n^* = f_n(X,Y,U)
\end{equation}
holds, and that $Y_n$ and $Y_n^*$ satisfy
\begin{eqnarray*}
&& \PROB\{ Y_n=y_{C_{i,n}},Y_n^*=y_{C_{j,n}}|X=x\} \\
&& =
\begin{cases}
\PROB\{Y \in C_{i,n}|X=x\}, & i=j \mbox{ and } \PROB\{Y \in C_{i,n}\}
\geq \PROB\{Y \in C_{i,n}|X=x\}\\
0 &  i \neq j \mbox{ and } \PROB\{Y \in C_{i,n}\}
\geq \PROB\{Y \in C_{i,n}|X=x\} \\
\frac{(\mu_j-\lambda_j)\cdot (\lambda_i-\mu_i)}{a},
& j \leq k \mbox{ and } \PROB\{Y \in C_{i,n}\}
< \PROB\{Y \in C_{i,n}|X=x\} \\
\PROB\{ Y \in C_{i,n}\}, & i=j \mbox{ and } \PROB\{Y \in C_{i,n}\}
< \PROB\{Y \in C_{i,n}|X=x\} \\
0, & \mbox{else}.
\end{cases}
\end{eqnarray*}
Here the third and the fourth rows in the right-hand side above
do not contradict each other since $\PROB\{Y \in C_{i,n}\}
< \PROB\{Y \in C_{i,n}|X=x\} $ implies $i>k$.

Next we show
\begin{equation}
\label{plea1eq3}
\PROB\{Y_n^* = y_{C_{j,n}}|X=x\} = \PROB\{ Y_n=y_{C_{j,n}} \}
\quad
\mbox{for all } x, \mbox{ all } j=1, \dots, N_n.
\end{equation}
Fix $x \in \R^N$. If
\[
\PROB\{ Y \in C_{j,n} \} < \PROB\{ Y \in C_{j,n} | X=x \},
\]
then $j \geq k+1$ and
\begin{eqnarray*}
&&
\PROB\{Y_n^* = y_{C_{j,n}}|X=x\} = 
\sum_{i=1}^{N_n}
\PROB\{Y_n=y_{C_{i,n}}, Y_n^* = y_{C_{j,n}}|X=x\} \\
&&
=
\sum_{i=1}^{k}
\PROB\{Y_n=y_{C_{i,n}}, Y_n^* = y_{C_{j,n}}|X=x\} 
+
\PROB\{Y_n=y_{C_{j,n}}, Y_n^* = y_{C_{j,n}}|X=x\} 
\\
&&
\quad
+
\sum_{i=k+1, \dots, N_n, i \neq j}
\PROB\{Y_n=y_{C_{i,n}}, Y_n^* = y_{C_{j,n}}|X=x\} 
\\
&&
=
0 + \PROB\{Y \in C_{j,n} \} + 0 
=
\PROB\{ Y_n = y_{ C_{j,n}} \} 
.
\end{eqnarray*}

If
\[
\PROB\{ Y \in C_{j,n} \} \geq \PROB\{ Y \in C_{j,n} | X=x \},
\]
then $j \leq k$ and
\begin{eqnarray*}
&&
\PROB\{Y_n^* = y_{C_{j,n}}|X=x\} = 
\sum_{i=1}^{N_n}
\PROB\{Y_n=y_{C_{i,n}}, Y_n^* = y_{C_{j,n}}|X=x\} \\
&&
=
\PROB\{Y_n=y_{C_{j,n}}, Y_n^* = y_{C_{j,n}}|X=x\} 
+
\sum_{i=1, \dots, k, i \neq j}
\PROB\{Y_n=y_{C_{i,n}}, Y_n^* = y_{C_{j,n}}|X=x\} 
\\
&&
\quad
+
\sum_{i=k+1}^{N_n}
\PROB\{Y_n=y_{C_{i,n}}, Y_n^* = y_{C_{j,n}}|X=x\} 
\\
&&
=
\PROB\{Y_n \in {C_{j,n}} | X=x \}
+ 0 +
\sum_{i=k+1}^{N_n}
\frac{(\mu_j-\lambda_j)\cdot (\lambda_i-\mu_i)}{a}
\\
&&
=
\lambda_j + (\mu_j-\lambda_j) \cdot \frac{1}{a} \cdot \sum_{i=k+1}^{N_n}
   (\lambda_i-\mu_i)\\
&&
=
\lambda_j + (\mu_j-\lambda_j) \cdot \frac{1}{a} \cdot a = \mu_j = \PROB\{Y_n=y_{C_{j,n}}\}.
\end{eqnarray*}
This proves (\ref{plea1eq3}).

From (\ref{plea1eq3}) we conclude
\begin{eqnarray*}
\PROB\{ Y_n^* = y_{C_{j,n}} \}
&=&
\int \PROB\{ Y_n^* = y_{C_{j,n}} | X=x \} \PROB_X (dx)
=
\int \PROB\{ Y_n = y_{C_{j,n}} \} \PROB_X (dx)
\\
&=& \PROB\{ Y_n = y_{C_{j,n}} \} 
\end{eqnarray*}
for $j=1, \dots, N_n$, hence
\begin{equation}
\label{plea1eq4}
\PROB_{Y_n^*} = \PROB_{Y_n}.
\end{equation}
Furthermore, for $j \in \{1, \dots, N_n\}$ and $A \in \B_N$
we can conclude from (\ref{plea1eq3}) and (\ref{plea1eq4})
\begin{eqnarray*}
\PROB\{ Y_n^* = y_{C_{j,n}}, X \in A \}
&=&
\int_A \PROB\{ Y_n^* = y_{C_{j,n}} | X=x \} \PROB_X (dx)
\\
&=&
\int_A \PROB\{ Y_n = y_{C_{j,n}} \} \PROB_X (dx)
\\
&=&
 \PROB\{ Y_n = y_{C_{j,n}} \}  \cdot \PROB\{ X \in A \}
\\
&=&
\PROB\{ Y_n^* = y_{C_{j,n}} \}  \cdot \PROB\{ X \in A \},
\end{eqnarray*}
hence
\begin{equation}
\label{plea1eq5}
Y_n^* \mbox{ and } X \mbox{ are independent}.
\end{equation}

Set
\[
\C_n = \left\{
\cup_{j \in J} C_{j,n} \, : \, J \subseteq \{1, \dots, N_n\}
\right\}
\quad \mbox{and} \quad
\C = \cup_{n \in \N} \, \C_n.
\]
By the definition of $Y_n^*$ we get
\begin{eqnarray*}
&&
\PROB \{ Y_n \neq Y_n^* | X=x \}
\\
&&
=
\sum_{1 \leq i,j \leq N_n, i \neq j}
\PROB \{ Y_n = y_{C_{i,n}},  Y_n^*= y_{C_{j,n}} | X=x \}
\\
&&
= \sum_{j=1}^k \sum_{i=k+1}^{N_n}
\frac{(\mu_j-\lambda_j)\cdot (\lambda_i-\mu_i)}{a}
\\
&&
= \frac{1}{a} \cdot a \cdot a = a
\\
&&
=
\sum_{l=1}^k \left(
\PROB\{ Y \in C_{l,n}\} - \PROB\{ Y \in C_{l,n} | X=x \}
\right)
\\
&&
=
\PROB\{ Y \in \cup_{l=1}^k C_{l,n}\} - \PROB\{ Y \in \cup_{l=1}^k C_{l,n} | X=x \}
\\
&&
\leq
\sup_{C \in \C_n}
\left|
\PROB\{ Y \in C\} - \PROB\{ Y \in C | X=x \}
\right|
\end{eqnarray*}
where we have used
\[
\cup_{l=1}^k C_{l,n}  \in \C_n.
\]
(Here the sets in the union above depend on $x$ and not only on $l$
and $n$). Hence
\begin{eqnarray}
\label{plea1eq6}
\PROB \{ Y_n \neq Y_n^* \}
&=& \int \PROB \{ Y_n \neq Y_n^* | X=x \} \PROB_X(dx)
\nonumber \\
&\leq&
\int 
\sup_{C \in \C_n}
\left|
\PROB\{ Y \in C\} - \PROB\{ Y \in C | X=x \}
\right|
 \PROB_X(dx)
\nonumber \\
&=&
\EXP \left\{
\sup_{C \in \C_n}
\left|
\PROB\{ Y \in C\} - \PROB\{ Y \in C | X\}
\right|
\right\}
\nonumber \\
&\leq&
\EXP \left\{
\sup_{C \in \C}
\left|
\PROB\{ Y \in C\} - \PROB\{ Y \in C | X\}
\right|
\right\}.
\end{eqnarray}

For $M \in \N$ we have
\begin{eqnarray*}
&&
\sup_{n \in \N}
\PROB \left\{
(X,Y,Y_n,Y_n^*,Z,U)
\notin
[-M,M]^{
N + 3d + K + 1
}
\right\}
\\
&& \leq
\PROB\{ X \notin [-M,M]^N \}
+
\PROB\{ Y \notin [-M,M]^d \}
+
\sup_{n \in \N}
\PROB\{ Y_n \notin [-M,M]^d \}
\\
&&
\quad
+
\sup_{n \in \N}
\PROB\{ Y_n^* \notin [-M,M]^d \}
+
\PROB\{ Z \notin [-M,M]^K \}
+
\PROB\{ U \notin [-M,M] \}
\\
&&
\rightarrow 0 \quad (M \rightarrow \infty),
\end{eqnarray*}
since (\ref{plea1eq4}) and the definition of $Y_n$ imply
\begin{eqnarray*}
&&
\sup_{n \in \N}
\PROB\{ Y_n^* \notin [-M,M]^d \} = 
\sup_{n \in \N}
\PROB\{ Y_n \notin [-M,M]^d \}
\\
&&
\leq
\sup_{n \in \N: 2^n+1 \leq M}
\PROB\{ Y_n \notin [-M,M]^d \}
+\sup_{n \in \N: M<2^n+1}
\PROB\{ Y_n \notin [-M,M]^d \}
\\
&&
= 0 + \PROB\{ Y \notin [-M,M]^d \}
\rightarrow 0 \quad (M \rightarrow \infty). 
\end{eqnarray*}
From this we get that
\[
\left(
\PROB_{
(X,Y,Y_n,Y_n^*,Z,U)
}
\right)_{n \in \N}
\]
is tight, hence the theorem of Prokhorov implies that there exists
a subsequence $(n_k)_k$ of $(n)_n$ and a probability measure
\[
\PROB_{
(
\tilde{X},
\tilde{Y },
\tilde{\tilde{Y } },
\tilde{Y }^*,
\tilde{Z },
\tilde{ U}
)
}
\] 

such that
\begin{equation}
\label{plea1eq7}
\PROB_{
(X,Y,Y_{n_k},Y_{n_k}^*,Z,U)
}
\rightarrow
\PROB_{
(
\tilde{X},
\tilde{Y },
\tilde{\tilde{Y } },
\tilde{Y }^*,
\tilde{Z },
\tilde{ U}
)
}
\quad
\mbox{weakly}.
\end{equation}

By the Portmanteau theorem we can conclude from this
\begin{eqnarray*}
&&
\PROB\{  \tilde{Y } \neq
\tilde{\tilde{Y}} \}
\}
= \limsup_{\epsilon \rightarrow 0}
\PROB\{  |\tilde{Y} -
\tilde{\tilde{Y}}| > \epsilon \}
\leq
\limsup_{\epsilon \rightarrow 0}
\liminf_{k \rightarrow \infty}
\PROB\{  |Y - Y_{n_k} | > \epsilon \}
=0,
\end{eqnarray*}
where the last equality follows from (\ref{plea1eq2}). Hence we can
conclude from (\ref{plea1eq7})
\begin{equation}
\label{plea1eq8}
\PROB_{
(X,Y,Y_{n_k},Y_{n_k}^*,Z,U)
}
\rightarrow
\PROB_{
(
\tilde{X},
\tilde{Y },
\tilde{Y } ,
\tilde{Y }^*,
\tilde{Z },
\tilde{ U}
)
}
\quad
\mbox{weakly},
\end{equation}
By the continuous mapping theorem this implies
\[
\PROB_{
(X,Y,Z,U)
}
\rightarrow
\PROB_{
(
\tilde{X},
\tilde{Y },
\tilde{Z },
\tilde{ U}
)
}
\quad
\mbox{weakly}.
\]
and because of the uniqueness of the limit distribution for weak
convergence we get

\begin{equation}
\label{plea1eq9}
\PROB_{
(X,Y,Z,U)
}
=
\PROB_{
(
\tilde{X},
\tilde{Y },
\tilde{Z },
\tilde{ U}
)
}.
\end{equation}

By the proof of Skorokhod's representation theorem in Skorokhod (1978)
(cf., Lemma \ref{lea4} below)
we can conclude from (\ref{plea1eq8}) that there exists a probability
space
$(\bar{\Omega},\bar{\A}, \bar{\PROB})$
and random variables
$(\bar{X},\bar{Y},\bar{Y}_{n_k},\bar{Y}_{n_k}^*,\bar{Z},\bar{U})$
and
$(\bar{X},\bar{Y},\bar{Y},\bar{\bar{Y}}^*,\bar{Z},\bar{U})$
such that

\begin{equation}
\label{plea1eq10}
\PROB_{
(X,Y,Y_{n_k},Y_{n_k}^*,Z,U)
}
=
\bar{\PROB}_{
(\bar{ X}, \bar{ Y}, \bar{ Y}_{n_k}, \bar{ Y}_{n_k}^*, \bar{ Z}, \bar{U})
},
\end{equation}

\begin{equation}
\label{plea1eq11}
\PROB_{
(
\tilde{X},
\tilde{Y },
\tilde{Y } ,
\tilde{Y }^*,
\tilde{Z },
\tilde{ U}
)
}
=
\bar{\PROB}_{
(
\bar{X},
\bar{Y },
\bar{\bar{Y} } ,
\bar{Y }^*,
\bar{Z },
\bar{ U}
)
}
\end{equation}
and
\begin{equation}
\label{plea1eq12}
(\bar{ X}, \bar{ Y}, \bar{ Y}_{n_k}, \bar{ Y}_{n_k}^*, \bar{ Z},U)
\rightarrow
(\bar{X},
\bar{Y },
\bar{\bar{Y} },
\bar{Y }^*,
\bar{Z },
\bar{ U}
)
\quad a.s.
\end{equation}
From (\ref{plea1eq11}) we get
\[
\bar{\PROB}\{ 
\bar{Y } \neq
\bar{\bar{Y} }\}
=
\PROB\{ \tilde{Y} \neq \tilde{Y} \} =0,
\]
hence (\ref{plea1eq11}) and (\ref{plea1eq12}) imply 
\begin{equation}
\label{plea1eq13}
\PROB_{
(
\tilde{X},
\tilde{Y },
\tilde{Y } ,
\tilde{Y }^*,
\tilde{Z },
\tilde{ U}
)
}
=
\bar{\PROB}_{
(
\bar{X},
\bar{Y },
\bar{Y}  ,
\bar{Y }^*,
\bar{Z },
\bar{ U}
)
}
\end{equation}
and
\begin{equation}
\label{plea1eq14}
 (\bar{ X}, \bar{ Y}, \bar{ Y}_{n_k}, \bar{ Y}_{n_k}^*, \bar{ Z},U)
\rightarrow
(
\bar{X},
\bar{Y },
\bar{Y} ,
\bar{Y }^*,
\bar{Z },
\bar{ U}
)
\quad a.s.
\end{equation}
We show next that 
$
\bar{X},
\bar{Y },
\bar{Y }^*,
\bar{Z },
\bar{ U}
$
satisfy (\ref{lea1eq1})-(\ref{lea1eq5}).

Identity  (\ref{lea1eq1}) follows from 
  (\ref{plea1eq13}) and (\ref{plea1eq9}), because these two equations
imply
\[
\bar{\PROB}_{
(
\bar{X},
\bar{Y },
\bar{Z },
\bar{ U}
)
}
=
\PROB_{
(
\tilde{X},
\tilde{Y },
\tilde{Z },
\tilde{ U}
)
}
=
\PROB_{
(X,Y,Z,U)
}
.
\]

Identity (\ref{lea1eq2}) follows from 
\begin{equation}
\label{plea1eq15}
\bar{\PROB}_{\bar{Y}_{n_k}}\rightarrow \bar{\PROB}_{\bar{Y}}
\quad \mbox{weakly}
\end{equation}
(cf., (\ref{plea1eq14})),
\begin{equation}
\label{plea1eq16}
\bar{\PROB}_{\bar{Y}^*_{n_k}} \rightarrow \bar{\PROB}_{\bar{Y}^*}
\quad \mbox{weakly}
\end{equation}
(cf., (\ref{plea1eq14})) and
\begin{equation}
\label{plea1eq17}
\bar{\PROB}_{\bar{Y}_{n_k}^*} =
{\PROB}_{{Y}_{n_k}^*} =
{\PROB}_{{Y}_{n_k}} =
 \bar{\PROB}_{\bar{Y}_{n_k}}
\end{equation}
(cf., (\ref{plea1eq10}) and (\ref{plea1eq4})),
which imply
\begin{equation}
\label{plea1eq18}
\bar{\PROB}_{\bar{Y}_{n_k}}
\rightarrow
\bar{\PROB}_{\bar{Y}^*}
\quad \mbox{weakly}.
\end{equation}
Using the uniqueness of the limit distribution we get 
(\ref{lea1eq2}) from (\ref{plea1eq18}), (\ref{plea1eq15}) and
(\ref{lea1eq1}).

In order to show
(\ref{lea1eq3})
we observe that
\[
\bar{\PROB} \{ \bar{Y}_{n_k}^* = f_{n_k}(\bar{X},\bar{Y},\bar{U}) \}
=
{\PROB} \{ {Y}_{n_k}^* = f_{n_k}({X},{Y},{U}) \}
=1
\]
(cf., (\ref{plea1eq10}) and (\ref{plea1eq2})) and (\ref{plea1eq14})
imply that we have with probability one
\[
\bar{Y}^*
=
\lim_{k \rightarrow \infty} \bar{Y}_{n_k}^*
=
\lim_{k \rightarrow \infty} f_{n_k}(\bar{X},\bar{Y},\bar{U}).
\]
Set
\[
f(x,y,u) = \limsup_{k \rightarrow \infty} f_{n_k}(x,y,u)
\cdot 1_{\{ \limsup_{k \rightarrow \infty} f_{n_k}(x,y,u)<\infty\}}.
\]
Then $\bar{Y}^* < \infty$ $a.s.$ implies
\[
\bar{Y}^*
=
\lim_{k \rightarrow \infty} f_{n_k}(\bar{X},\bar{Y},\bar{U})
=
f (\bar{X},\bar{Y},\bar{U}),
\]
which proves (\ref{lea1eq3}).

Next we prove (\ref{lea1eq4}).
Let
$\varphi_{(\bar{X},\bar{Y}^*)}$, $\varphi_{\bar{X}}$ and $\varphi_{\bar{Y}^*}$
be the characteristic functions of
$(\bar{X},\bar{Y}^*)$, $\bar{X}$ and $\bar{Y}^*$, resp. Then
the continuity theorem of Levy-Cramer together with
(\ref{plea1eq13}), (\ref{plea1eq8}), (\ref{plea1eq5}) and (\ref{plea1eq10}) imply
\begin{eqnarray*}
&&
\varphi_{(\bar{X},\bar{Y}^*)}(u,v)
=
\varphi_{(\tilde{X},\tilde{Y}^*)}(u,v)
=
\lim_{k \rightarrow \infty}
\varphi_{(X,Y_{n_k}^*)}(u,v)
=
\lim_{k \rightarrow \infty}
\varphi_{X}(u) \cdot \varphi_{Y_{n_k}}(v)
\\
&&
=
\varphi_X(u) \cdot \varphi_{\tilde{Y}^*}(v)
=
\varphi_X(u) \cdot \varphi_{\bar{Y}^*}(v),
\end{eqnarray*}
hence  (\ref{lea1eq4}) holds.

We finish the proof by showing (\ref{lea1eq5}).
To do this we conclude from 
\[
\bar{\PROB}_{(\bar{Y}_{n_k},\bar{Y}^*_{n_k})}
\rightarrow \bar{\PROB}_{(\bar{Y},\bar{Y}^*)} \quad \mbox{weakly}
\]
(cf., (\ref{plea1eq14})), the Portmanteau theorem, (\ref{plea1eq10}) and
(\ref{plea1eq6})
\begin{eqnarray*}
\bar{\PROB} \{ \bar{Y} \neq \bar{Y}^* \}
&\leq&
\liminf_{k \rightarrow \infty}
\bar{\PROB} \{ \bar{Y}_{n_k} \neq \bar{Y}^*_{n_k} \}
\\
&=&
\liminf_{k \rightarrow \infty}
\PROB\{ Y_{n_k} \neq Y^*_{n_k} \}
\\
& \leq &
\EXP \left\{
\sup_{C \in \C}
\left|
\PROB\{ Y \in C\} - \PROB\{ Y \in C | X\}
\right|
\right\},
\end{eqnarray*}
hence (\ref{lea1eq5}) holds. \hfill $\Box$

In the proof above we have used the following modification
of the classical representation theorem of Skorokhod.

\begin{lemma}
\label{lea4}
Let $M,N \in \N$, let $(\Omega,\A,\PROB)$ be a probability
space, and let
$X: \bar\Omega \rightarrow \R^M$, $Y, Y_n: \bar\Omega \rightarrow \R^N$
be Borel measurable random variables with
\[
\PROB_{(X,Y_n)} \rightarrow \PROB_{(X,Y)} \quad \mbox{weakly}.
\]
Then there exists  a probability space
$(\bar \Omega,\bar \A,\bar \PROB)$ 
and Borel measurable random variables
$\bar X: \Omega \rightarrow \R^M$, $\bar Y, \bar Y_n: \Omega \rightarrow \R^N$
such that
\begin{equation}
\label{lea4eq1}
\PROB_{(X,Y_n)}=\bar \PROB_{(\bar X, \bar Y_n)}\quad (n \in \N),
\end{equation}
\begin{equation}
\label{lea4eq2}
\PROB_{(X,Y)}=\bar \PROB_{(\bar X, \bar Y)}
\end{equation}
and
\begin{equation}
\label{lea4eq3}
(\bar X, \bar Y_n) \rightarrow (\bar X, \bar Y) \quad a.s.
\end{equation}
\end{lemma}

\noindent
{\bf Proof.} The assertion follows from Skorokhod (1978), pp. 10-11,
if we use the construction there to define
$S_{i_1, \dots, i_k}$ and $\tilde S_{j_1, \dots, j_k}$ for
$\R^M$ and $\R^N$, resp., use then as there
\[
S_{i_1, \dots, i_k} \times \tilde S_{j_1, \dots, j_k}
\]
to construct 
$(\bar{X}_n^{(k)}, \bar{Y}_n^{(k)})$ from $(X,Y)$, and use that on
\[
\Delta^{(n)}_{i_1, \dots, i_k, j_1, \dots, j_k}
\]
$\bar{X}_n^{(k)}$ has the same value for all $j_1$, \dots, $j_k$.
Let $\lambda$ denote the Lebesgue measure.
Since we have that
\begin{eqnarray*}
\sum_{j_1, \dots, j_k} \lambda(\Delta^{(n)}_{i_1, \dots, i_k, j_1, \dots, j_k})
&
= &
\sum_{j_1, \dots, j_k}\PROB\{
X \in S_{i_1, \dots, i_k}, Y_n \in \tilde{S}_{j_1, \dots, j_k}
\}
\\
&
= &
\PROB\{
X \in S_{i_1, \dots, i_k} \}
\end{eqnarray*}
does not depend on $n$ and hence
\[
\bar{X}_n(\omega)
=
\lim_{k \rightarrow \infty} \bar{X}_n^{(k)} (\omega)
\]
does not depend on $n$ for all $\omega$, the construction in the
proof of Skorokhod (1978), pp. 10-11, gives us that the random
variables $(\bar{X}_n, \bar{Y}_n)$ constructed there with the property
\[
(\bar{X}_n, \bar{Y}_n) \rightarrow (\bar{X}, \bar{Y}) \quad a.s.
\]
satisfy
\[
\bar{X}_n = \bar{X}
\]
for all $n \in \N$. \hfill $\Box$

Now we are ready to prove Lemma \ref{lea2}.

\noindent
{\bf Proof of Lemma \ref{lea2}.}
It suffices to show that for every $l \in \{1,\dots,n\}$
there exists a probability space
$(\bar{\Omega},\bar{\A},\bar{\PROB})$ and random variables
$\bar{ Z}_0$, $\bar{ Z}_1$, \dots, $\bar{ Z}_n$,
$\bar{ Z}_1^*$, \dots, $\bar{ Z}_l^*$, $\bar{U}_1$, \dots, $\bar{U}_l$
defined on this probability space
such that
\begin{equation}
\label{plea2eq1}
\PROB_{(Z_0,Z_1, \dots, Z_n)}
=
\bar{\PROB}_{(\bar{ Z}_0, \bar{ Z}_1, \dots, \bar{ Z}_n)},
\end{equation}
\begin{equation}
\label{plea2eq2}
\bar{ Z}_0, \bar{ Z}_1^*, \dots, \bar{ Z}_l^*\quad \mbox{are } i.i.d.,
\end{equation}
\begin{equation}
\label{plea2eq3}
\bar{Z}_k^*=f_k(\bar{ U}_1, \dots, \bar{
  U}_{k-1},\bar{ Z}_0, \bar{ Z}_1, \dots, \bar{ Z}_{k-1}, \bar{
  Z}_1^*, \dots, \bar{ Z}_{k-1}^*, \bar{Z}_k,\bar{U}_k)
\end{equation}
for some Borel measurable function $f_k$ and
all $k \in \{1, \dots, l\}$,
\begin{equation}
\label{plea2eq5}
(\bar{ Z}_0, \bar{ Z}_1, \dots, \bar{ Z}_n) \mbox{ and }
(\bar{ U}_1, \dots, \bar{
  U}_{l})
\mbox{ are independent},
\end{equation}
and
\begin{eqnarray}
\label{plea2eq4}
&&
\bar{\PROB}\{ \exists k \in \{1,\dots, l\}: \bar{Z}_k \neq \bar{Z}^*_k
\}
\nonumber \\
&&
\leq
(l-1) \cdot \max_{s \in \{2, \dots, l\}}
\EXP \left\{
\sup_{C \in \C}
\left|
\PROB\{ Z_s \in C | Z_1, \dots, Z_{s-1} \} - \PROB\{Z_s \in C\}
\right| \right\}.
\end{eqnarray}

We show this by induction on $l$. The assertion trivially holds for
$l=1$ if we choose
\[
(\bar{\Omega},\bar{\A},\bar{\PROB})=({\Omega},{\A},{\PROB}),
\quad \bar{Z}_i=Z_i \quad (i=0, \dots, n)
\quad \mbox{and} \quad
\bar{Z}_1^*=Z_1.
\]
So assume next that the assertion holds for some $l \in \{1, \dots,
n-1\}$.
We extend the probability space
$(\bar{\Omega},\bar{\A},\bar{\PROB})$
such that there exists a random variable $U_{l+1}$ uniformly distributed on $[0,1]$,  
 which is independent from
$\bar{ Z}_0$, $\bar{ Z}_1$, \dots, $\bar{ Z}_n$,
$\bar{ Z}_1^*$, \dots, $\bar{ Z}_l^*$, $\bar{U}_1$, \dots,
$\bar{U}_l$.
To simplify the notation we write in the sequel $\PROB$ and
${ Z}_0$, ${ Z}_1$, \dots, ${ Z}_n$,
${ Z}_1^*$, \dots, ${ Z}_l^*$, ${U}_1$, \dots,
${U}_l$
instead of
$\bar{\PROB}$ and
$\bar{ Z}_0$, $\bar{ Z}_1$, \dots, $\bar{ Z}_n$,
$\bar{ Z}_1^*$, \dots, $\bar{ Z}_l^*$, $\bar{U}_1$, \dots,
$\bar{U}_l$. We will apply Lemma \ref{lea1} in order to show the
assertion for $l+1$. To do this, we set
\[
X =({ U}_1, \dots, {
  U}_{l},
{ Z}_0, { Z}_1, \dots, { Z}_{l}, {
  Z}_1^*, \dots, { Z}_{l}^*),
\]
\[
Y=Z_{l+1},
\]
\[
Z=(Z_{l+2}, \dots, Z_n)
\]
and
\[
U=U_{l+1}
\]
Application of Lemma \ref{lea1} yields
\[
\bar{X} =(\bar{ U}_1, \dots, \bar{
  U}_{l},\bar{ Z}_0, \bar{ Z}_1, \dots, \bar{ Z}_{l}, \bar{
  Z}_1^*, \dots, \bar{ Z}_{l}^*),
\]
\[
\bar Y=\bar Z_{l+1},
\]
\[
\bar Y^*=\bar Z_{l+1}^*,
\]
\[
\bar Z=(\bar Z_{l+2}, \dots, \bar Z_n)
\]
and
\[
\bar U=\bar U_{l+1}
\]
such that (\ref{lea1eq1})--(\ref{lea1eq4}) hold.
We show next that
$\bar{ Z}_0$, $\bar{ Z}_1$, \dots, $\bar{ Z}_n$,
$\bar{ Z}_1^*$, \dots, $\bar{ Z}_{l+1}^*$, $\bar{U}_1$, \dots, $\bar{U}_{l+1}$
satisfy (\ref{plea2eq1})--(\ref{plea2eq4}) for $l+1$ instead of $l$.

(\ref{plea2eq1}) directly follows from (\ref{lea1eq1}) and our choice
of $X$, $Y$ and $Z$. 

Because of the induction hypothesis and
(\ref{plea2eq1}),
(\ref{plea2eq2}) holds for $l$, and using (\ref{lea1eq2}) and
(\ref{lea1eq4})
we can conclude from this that 
(\ref{plea2eq2}) also holds for $l+1$.

For $k \leq l$ the induction hypothesis and  (\ref{lea1eq1}) imply
(\ref{plea2eq3}). For $k=l+1$ identity
(\ref{plea2eq3}) is a direct consequence of (\ref{lea1eq3}).

(\ref{plea2eq5}) follows from the induction hypothesis
(which together with (\ref{lea1eq1}) implies that (\ref{plea2eq5})
holds
for $l$) and $U_{l+1}$ independent from $Z_0$, \dots, $Z_n$, $U_1$,
\dots, $U_l$  (which together with (\ref{lea1eq1}) implies  
$\bar U_{l+1}$ independent from $\bar Z_0$, \dots, $\bar Z_n$, $\bar U_1$,
\dots, $\bar U_l$).

So it remains to show (\ref{plea2eq4}) for $l+1$. Because of the
induction hypothesis, (\ref{plea2eq1}) and the union bound it suffices to show
\[
\bar{\PROB}\{ \bar{Z}_{l+1} \neq \bar{Z}^*_{l+1}
\}
\leq
\EXP \left\{
\sup_{C \in \C}
\left|
\PROB\{ Z_{l+1} \in C | Z_1, \dots, Z_{l} \} - \PROB\{Z_{l+1} \in C\}
\right| \right\}.
\]
(\ref{lea1eq5}) implies
\begin{eqnarray*}
&&
\bar{\PROB}\{ \bar{Z}_{l+1} \neq \bar{Z}^*_{l+1}\}
=
\bar{\PROB}\{ \bar{Y} \neq \bar{Y}^*\}
\leq
\EXP
\left\{
 \sup_{C \in \C}
\left|
\PROB\{  Y \in C | X \}
-
\PROB\{ Y \in C \}
\right|
\right\}
\\
&&
=\EXP
\left\{
\sup_{C \in \C}
\left|
\PROB\{  Z_{l+1} \in C | { U}_1, \dots, {
  U}_{l}, { Z}_0, { Z}_1, \dots, { Z}_{l}, {
  Z}_1^*, \dots, { Z}_{l}^*\}
-
\PROB\{ Z_{l+1} \in C \}
\right|
\right\}
\end{eqnarray*}
By (\ref{plea2eq3}) we get that the $\sigma$ field generated by
${ U}_1, \dots, {
  U}_{l}, { Z}_0, { Z}_1, \dots, { Z}_{l}, {
  Z}_1^*, \dots, { Z}_{l}^*$ is the same as the $\sigma$-field generated by 
$  { U}_1, \dots, {
  U}_{l}, { Z}_0, { Z}_1, \dots, { Z}_{l}$, hence
\begin{align*}
&\PROB\{  Z_{l+1} \in C | { U}_1, \dots, {
  U}_{l}, { Z}_0, { Z}_1, \dots, { Z}_{l}, {
  Z}_1^*, \dots, { Z}_{l}^*\}\\
&=
\PROB\{ Z_{l+1} \in C |   { U}_1, \dots, {
  U}_{l}, { Z}_0, { Z}_1, \dots, { Z}_{l}\}.
\end{align*}
Because of $( { U}_1, \dots, {
  U}_{l})$ independent from $(Z_{l+1}, { Z}_0, { Z}_1, \dots, {
  Z}_{l})$
the last conditional probability is equal to
\[
\PROB\{ Z_{l+1} \in C | { Z}_0, { Z}_1, \dots, { Z}_{l}\},
\]
from which we can conclude
\begin{eqnarray*}
&&
\bar{\PROB}\{ \bar{Z}_{l+1} \neq \bar{Z}^*_{l+1}\}
\\
&&
\leq \EXP
\left\{
\sup_{C \in \C}
\left|
\PROB\{  Z_{l+1} \in C |  { U}_1, \dots, {
  U}_{l}, { Z}_0, { Z}_1, \dots, { Z}_{l}, {
  Z}_1^*, \dots, { Z}_{l}^*\}
-
\PROB\{ Z_{l+1} \in C \}
\right|
\right\}
\\
&&
=
\EXP
\left\{
\sup_{C \in \C}
\left|
\PROB\{  Z_{l+1} \in C | { Z}_0, { Z}_1, \dots, { Z}_{l}
  \}
-
\PROB\{ Z_{l+1} \in C \}
\right|
\right\},
\end{eqnarray*}
which proves (\ref{plea2eq4}).
\hfill $\Box$

\end{appendix}

\end{document}